\documentclass{amsart}
\usepackage[T1]{fontenc} % to get rid of  Latex warning
\usepackage{tabularx}
\usepackage{amssymb, enumerate}
\usepackage{amsthm}
\usepackage{amsmath}
\usepackage{xcolor}
 \usepackage{color}

\usepackage[matrix, arrow]{xy}
\xyoption{arrow}
\theoremstyle{plain}\newtheorem{Theorem}{Theorem}[section]
\theoremstyle{plain}\newtheorem{Conjecture}[Theorem]{Conjecture}
\theoremstyle{plain}
\theoremstyle{plain}\newtheorem{Lemma}[Theorem]{Lemma}
\theoremstyle{plain}\newtheorem{Proposition}[Theorem]{Proposition}
\theoremstyle{definition}
\theoremstyle{definition}
\theoremstyle{definition}
\theoremstyle{definition}
\theoremstyle{definition}
\theoremstyle{definition}
\theoremstyle{definition}

\def\CF{{\mathcal{F}}}

\def\CN{{\mathcal{N}}}

\def\CU{{\mathcal{U}}}

\def\CW{{\mathcal{W}}}

\def\F{{\mathbb F}}

\def\Aut{\mathrm{Aut}}    
\def\Frob{\mathrm{Frob}}                
\def\BM{\mathrm{BM}}

\def\dim{\mathrm{dim}}

\def\GL{\mathrm{GL}}
\def\HN{\mathrm{HN}}
\def\SL{\mathrm{SL}}

\def\Ind{\mathrm{Ind}}

\def\Inn{\mathrm{Inn}}

\def\Irr{\mathrm{Irr}}           
\def\Ker{\mathrm{Ker}}  
\def\Ly{\mathrm{Ly}}

\def\Out{\mathrm{Out}}

\def\Stab{\mathrm{Stab}} 

\newcommand{\m}{\textbf{m}}

\newcommand{\w}{\textbf{w}}
\renewcommand{\b}{\textbf{b}}

\makeindex
\title{Weight conjectures  for Parker--Semeraro fusion systems} 
\author{Radha Kessar}
\address{Department of Mathematics, University of Manchester, Manchester M13 9PL,
  United Kingdom}
\email{radha.kessar@manchester.ac.uk}
\author{Jason Semeraro}
\address{Department of Mathematics, Loughborough University, Loughborough LE11 3TU,
  United Kingdom}
\email{j.p.semeraro@lboro.ac.uk}
\author{Patrick Serwene}
\address{Technische Universität Dresden, Faculty of Mathematics, 01062 Dresden, Germany}
\email{patrick.serwene@tu-dresden.de}
\author{\.{I}pek Tuvay}
\address{Department of Mathematics, 
Mimar Sinan Fine Arts University, 
34380 \.{I}stanbul, 
Turkey
}
\email{ipek.tuvay@msgsu.edu.tr}

\date{\today}

\begin{document}

\begin{abstract}
We prove that  the Parker--Semeraro   systems   satisfy  six  of the nine  Kessar--Linckelmann--Lynd--Semeraro weight conjectures  for  saturated fusion systems. As a by-product  we obtain that Robinson's ordinary weight conjecture  holds for the principal $3$-block of  $\Aut(G_2(3))$, the principal $5$-blocks  of $HN$, $BM$, $\Aut(HN)$, $Ly$, the principal $7$-block of $M$, and the principal $p$-blocks of $G_2(p)$  for $p\geq 3 $.\end{abstract}
%{\bf Keywords:} Fusion systems,  weights.

\maketitle

\section{Introduction}

Let $p$ be a  prime number, $S $ a  finite  $p$-group and  $\CF$  a  saturated fusion  system on $S$.   To  each  
$\CF$-compatible family  $\alpha $, and   each  non-negative integer  $d$   are associated   integers  $  {\w } (\CF,  \alpha) $,   $\m(\CF,  \alpha) $ and  $\m(\CF,  \alpha, d) $   which are the subject of a   collection of nine conjectures, together referred to as ``weight conjectures'',   put forward in  \cite{KLLS} (see  Section 2  for definitions).    Here we are concerned with  six of these:

\begin{Conjecture} \label{conj:weight}  (Kessar--Linckelmann--Lynd--Semeraro).  With the notation above, suppose that $|S|=p^a $.
\begin{enumerate}[\rm(i)]
\item     $\m(\CF,  \alpha)  \leq p^a $.
\item  $\w(\CF,  \alpha)  \leq p^s $  where $s$ is the sectional rank of $S$.
\item  $ \m(\CF, \alpha, d)  \geq 0 $   for all $ d \geq  0 $.
\item   If  $S$ is non-abelian, then $\m(\CF,0, d) \ne 0 $  for  some $d \ne a$.
\item  Let $r > 0$ be the smallest
positive integer such that S has a character of degree $p^r$. Then $r$ is the smallest positive
integer such that $\m(\mathcal F, \alpha, a - r) \neq 0.$
\item    $\m(\CF, \alpha) / \w( \CF, \alpha)  $  is  at  most the number of conjugacy classes of  $S$ and 
  $\m(\CF, \alpha) /\m( \CF, \alpha, a)  $ is at most  the number of conjugacy classes of $[S,S] $.
  \end{enumerate} 
\end{Conjecture}
Statement (i)   is Conjecture 2.1,  (ii) is  Conjecture 2.2, (iii)  is Conjecture 2.5,  (iv) is Conjecture 2.8, (v) is Conjecture 2.9 and  (vi)  is Conjecture 2.10  of \cite{KLLS}.   Note that in \cite{KLLS},   Conjectures 2.1 and 2.10  are    stated  with  the invariant ${\bf k}(\CF, \alpha) $ in  place  of $\m(\CF, \alpha)$, the relationship between  the two versions  is  discussed  at the end of this introduction.

Conjecture~\ref{conj:weight} is related to the well-known local-global counting  conjectures in the representation theory of  finite groups.   If   $B$ is a  $p$-block of  a  finite group $G$ with defect group $S$, and  $\alpha $  the associated   family of  K\"ulshammer--Puig  classes,  then  Alperin's weight conjecture, see \cite{Al}, is the   assertion that  $\w(\CF, \alpha) $ equals the number of irreducible Brauer characters of   $G$  in $B$  (see \cite[Proposition~5.4]{Ke}) and Robinson's ordinary weight conjecture, see \cite{Ro}, is  the  assertion    that  for each  $d\geq 0 $, $\m(\CF, \alpha,d) $ is the number of ordinary irreducible characters of $G$ of $p$-defect $d$  in  $B$ (see \cite[IV.5.49]{AKO}). Consequently,  when  $(S, \CF, \alpha)$  arises   from a  block of a finite group then   Conjecture \ref{conj:weight}   translates to  various   local-global   statements (such as Brauer's $k(B)$-conjecture, Brauer's height zero conjecture,  Eaton-Moreto,  Malle-Navarro  and Malle-Robinson conjectures). We refer the reader to \cite{KLLS} for  a discussion  of these connections.      

The interest in  Conjecture~\ref{conj:weight}   beyond the  local-global counting conjectures   comes from the fact that  it applies   equally well  in the situation where  $(S, \CF, \alpha) $  does  not arise from a  $p$-block,  and   in particular  when $\CF $ is  block-exotic.  
It was shown to  hold  for the    exotic Ruiz--Viruel  fusion systems  in   \cite{KLLSa} and   has been verified  for  the  Benson--Solomon fusion systems (see \cite{LS, S23}).   The aim of this article is to provide further evidence for exotic  systems.
 
Suppose  that $p$ is odd and    $S$   is   a Sylow  $p$-subgroup   of  $G_2(p) $.     In \cite{PS18}, Parker and Semeraro  classified  all  saturated fusion systems  $\CF$ on $S$   satisfying $O_p(\CF)=1 $.    We henceforth refer  to these fusion systems   as  {\it Parker--Semeraro} systems.  For  $p \ne 7 $, all  Parker--Semeraro systems   are realisable by finite groups    and hence for  principal $p$-blocks, whereas   for $p=7 $, there are $27$  exotic fusion systems  (out of a total of $29$).   It was shown  in \cite{Ser} that the   $27$ exotic   Parker--Semeraro  fusion  systems are also block-exotic, that is, they are not realisable  by blocks  of finite groups.     The following is  our first main result.  

\iffalse
\begin{Theorem} \label{thm:first}   Let $S$ be a Sylow $7$-subgroup  of $G_2(7)$  and   let $\CF$  be a saturated fusion system  on $S$ with $O_7(\CF) =1$.    Then Conjecture~\ref{conj:weight}  holds   when   $\alpha$   is the trivial family.
\end{Theorem}
\fi

\begin{Theorem}\label{thm:first_general} Suppose that  $p \ge 3$, let $S$ be a Sylow $p$-subgroup  of $G_2(p)$  and   let $\CF$  be a saturated fusion system  on $S$ with $O_p(\CF) =1$.    Then Conjecture~\ref{conj:weight}  holds   when   $\alpha$   is the trivial family.
\end{Theorem}

 A  question of  independent interest in the theory of fusion systems  is  that of determining, for a given  saturated fusion system,  the   possible families of compatible  families  $\alpha$.  We hope to address this question   for Parker--Semeraro  fusion systems  and  thereby extend our results  to cover all  $\alpha$  in future work.  
 Our  primary focus   is on   block-exotic  fusion  systems  but our calculations also  apply  to non-exotic  Parker--Semeraro  fusion  systems.   As a consequence  of these we obtain the  following result.

 \begin{Theorem}  \label{thm:second_general}   The principal  $p$-blocks of  $\Aut(G_2(3))$  $(p=3)$;  $\HN$, $\BM$, $\Aut(\HN)$, $\Ly$ $(p=5)$; $M$ $(p=7)$ and $G_2(p)$  $(p \ge 3)$    satisfy the   ordinary weight conjecture.
 \end{Theorem} 

We note that  by results of Eaton in \cite{Ea}, the ordinary weight conjecture is  equivalent to  Dade's projective conjecture  
(\cite{Da92},\cite{Da94})  in the sense that a  minimal counter-example  to one is a minimal counter-example to the  other. Further, Dade's   conjectures have been verified in  the cases considered in Theorem~\ref{thm:second_general} (\cite{huang},\cite{AnOBrien},\cite{AnWi04},\cite{AnWi10},\cite{SK03}).  However, given the nature of the equivalence  between the conjectures, these earlier results  do not directly imply Theorem~\ref{thm:second_general}.

The paper is organised as follows. Section~2  contains  some  preparatory  results  on counting weights in fusion  systems.   In Section~ 3,  we make some explicit   orbit-stabiliser calculations  for  the action of $\GL_2(p)$   on  $4$-dimensional   simple modules.
Section~4 contains the bulk of the calculations for the integers $\m(\CF, 0, d)$ and   $\w (\CF, 0) $, these  are summarised in  Propositions~\ref{prop:mvalues} and ~\ref{prop:wvalues}. The proofs  of Theorems~\ref{thm:first_general} and  ~\ref{thm:second_general}   are  completed in Section~5.    We make use  of a combination   of hand and computer  computations.     In particular, when  $p=3$,   there are two Parker-Semeraro systems, realised by   $G_2(3)$  and $\Aut(G_2(3))$. Here the groups are sufficiently small that an entirely computational treatment of Theorems \ref{thm:first_general}  and  ~\ref{thm:second_general} using \textsc{magma} is possible.

As we pointed out earlier and as  also  discussed  in \cite{KLLS}, parts (i) and (vi)  of  Conjecture~\ref{conj:weight} come in two versions - the one we  adopt here  and the one   with  the  integer ${\bf k}(\CF, \alpha)  $ in place of   
$\m(\CF, \alpha) $.   It is conjectured  that   that  these two integers  are  equal  (Conjecture~2.3 of  \cite{KLLS}).  Moreover, Theorems~1.1  and  1.2  of \cite{KLLS}  prove that this  equality  (for all saturated fusion systems) is a consequence of  the Alperin weight conjecture.

\section{Background on fusion systems.}
 Our  notation and terminology   on fusion systems  follows   \cite{AKO}.  Let $p$ be a prime  number, $S$ a finite $p$-group  and  $\CF$   a saturated fusion system on  $S$.  Let $ k $ be an algebraically closed field of  characteristic   $p$ and let  $\alpha $ be an $\CF$-compatible family    with  values in $k^{\times}$  (see \cite[Defn~4.1]{KLLS}).   If $\alpha $  is the trivial family, we denote it simply by $0$.  We refer to a  subgroup  of $S$ which  is both $\CF$-centric and $\CF$-radical  as an  $\CF$-centric-radical subgroup.
 For   an $\CF$-centric-radical subgroup 
$P$ of $S$,   denote by  $\CN_P$ to be the set of all non-empty normal chains $\sigma$ of $p$-subgroups of $\Out_{\CF}(P)$ starting at the trivial subgroup.  More  precisely these are the chains 
$$\sigma:= (1=X_0 <X_1<\ldots <X_m)$$
where $X_i \unlhd X_m$ for any $i=0, 1, \ldots, m$. Such a $\sigma$ is said to have length 
equal to $m$ and we write $|\sigma|=m$. Note that $\CN_P$ is an $\Out_{\CF}(P)$-stable set. 
Let 
$$\Irr^d(P)=\{\mu \in \Irr(P)\ | \ |P|/\mu(1)=p^d\}$$
be the set of ordinary irreducible characters of $P$ of defect $d$. For  $\sigma  \in  \CN_P$  and $\mu$ an irreducible character of $P$, let $I(\sigma)$ denote the stabiliser of $\sigma$ 
and  let  $I(\sigma, \mu)$ denote the stabiliser of the pair $(\sigma, \mu)$ 
under the action of $\Out_{\CF}(P)$. For an $\CF$-compatible family $\alpha$ and a non-negative integer $d$,  we set 
$$\mathbf w_P(\CF, \alpha, d)=\sum_{\sigma\in \CN_P/\Out_{\CF}(P)} (-1)^{|\sigma|} \sum_{\mu \in \Irr^d(P)/I(\sigma)} z(k_{\alpha} I(\sigma, \mu))$$
where $z(k_{\alpha} I(\sigma, \mu))$ denotes the number of simple and projective $k_{\alpha} I(\sigma, \mu)$-modules.
We set 
$$\m(\CF,\alpha,d):=\sum_{Q \in \CF^{cr}/\CF} \w_Q(\CF,\alpha,d);
$$ 
$$\m(\CF,\alpha):=\sum_{d \geq 0} \m(\CF,\alpha,d) 
$$
 and
 $$\w(\CF,\alpha):=\sum_{Q \in \CF^{cr}/\CF} z(k_{\alpha} \Out_\CF(Q)).$$
 Note that  in Section 2  of \cite{KLLS},  the  definitions above are made with the  sums running over $\CF^c $ but  by Lemma~4.1 and  Lemma~7.5 of \cite{KLLS}, running over  $\CF^{cr} $  yields the same numbers.

Lemma~\ref{lem:max}  and Lemma~\ref{weight:Q-S}  below     will be useful  in calculating $\w_P (\CF,0, d) $ in certain situations.

\begin{Lemma} \label{lem:max} Let $S$ be  a finite $p$-group,   $\CF$  a  saturated fusion system  on $S$ and $Q \leq  S$  a maximal  subgroup  which is    $\CF$-centric-radical.   Let $d $ be a  non-negative integer  and  let  $\Irr^d(Q)' $   be the subset of  $\Irr^d(Q)$  consisting of   those characters which are not $S$-stable.  Up to  $\Out_{\CF} (Q) $-conjugacy,  $\CN_Q $ has two elements,  namely  the trivial  chain $\sigma_1 =(1)$ and the chain  $\sigma_2=(1<   \Out_{S}  (Q)) $.     Further, $I(\sigma_2)  =  N_{\Out_{\CF}(Q) } (\Out_S(Q)) $ and  if $\chi \in \Irr(Q) $ is    $S$-stable, then  $z((k_{\alpha} I(\sigma_2, \chi)) =0$.   Consequently,    the contribution of  the   $\Out_{\CF}(Q)$-class of   $\sigma_2 $  to $\w_{Q} (\CF, 0, d)   $  equals
\[-\sum_{v \in   \Irr^d(Q)' /J }  z(k_{\alpha}\Stab_J(v) ) \]  where $J=  N_{\Out_{\CF} (Q) } (\Out_S(Q) )$.
 \end{Lemma}

\begin{proof}  Since  $Q$ is   maximal   (and hence normal)   in  $S$  and  since  $C_S(Q)   =Z(Q)$,   $\Out_S(Q) \cong C_p $ is  a  Sylow $p$-subgroup  of $\Out_{\CF}(S) $   from which the first  assertion   is   immediate.  The second assertion is immediate from the  definition of  $\sigma_2 $.  Suppose  $\chi \in \Irr (Q)$ is $
\CF$-stable, then  $\Out_S(Q)  \leq  I(\sigma_2, \chi )$. On the  other hand,   $\Out_S(Q)  $  is normal in  $I(\sigma_2) $. So $1\ne \Out_S(Q) \leq  O_p(  I(\sigma_2, \chi ))$  and  by  \cite[Lemma 4.11]{KLLS}   $z((k_{\alpha} I(\sigma_2, \chi)) =0$. The last assertion of the lemma is  now immediate. 
\end{proof}

\begin{Lemma}\label{lem:trivnontriv}  Let $S$   be a  finite $p$-group   and  $Q\leq   S$   a   characteristic subgroup with $C_S(Q) =Z(Q)$.   Let $\CF$ be a  saturated fusion system on  $S$.    Then restriction to  $Q$  induces  group isomorphisms
$ \Aut_{\CF}(S) /\Aut_{Z(Q)}(S)  \cong  N_{\Aut_{\CF} (Q) } (\Aut_S(Q) ) $ and  $ \Out_{\CF}(S)/\Out_{Q}(S)  \cong  N_{\Out_{\CF} (Q) } (\Out_S(Q))  $.
\end{Lemma}

\begin{proof}   Since $Q$ is characteristic  in $S$,  elements of 
$\Aut_{\CF}(S)$  restrict  to  elements of $\Aut_{\CF} (Q)$. By the extension axiom for saturated fusion systems, the image   of the resulting homomorphism from    $\Aut_{\CF}(S)$ to $ \Aut_{\CF} (Q)$  is $ N_{\Aut_{\CF}(Q) } (\Aut_S(Q) ) $ and by  Lemma  4.10 of \cite{LiFusion},  the kernel  is  $\Aut_{Z(Q) } (S) $.  This proves the   first isomorphism.  Now suppose that   
$\varphi \in  \Aut_{\CF}(S)   $ is such that $\varphi|_Q $ is an inner automorphism of   $Q$, say $\varphi |_Q = c_x $  for $x \in Q $. Then   letting $\varphi' \in \Aut_{\CF}(S) $ be  the  composition of $\varphi $ with  $c_{x^{-1}} $, we have that   $\varphi'|_Q =id_Q $. 
Hence $\varphi'  \in \Aut_{Z(Q)} (S) $  which implies that  $\varphi \in \Aut_{Q}(S) $. This proves the second isomorphism.
\end{proof}

 For a group $G$ acting on a  set $V$, denote by  $V/G$ the set of $G$-orbits of $V$. Note that if $H$ is  a  normal subgroup of $G$, then $V/H$ is  naturally a $G/H$-set.   Before proceeding, we record  the following elementary fact.
 
\begin{Lemma} \label{lem:fffree}Let $G$  be  a  group  acting on a set $V$ and  let $H$ be a normal subgroup of  $G$. Let $\bar V=V/H$ and  for  $v$ in  $V$ let  $\bar v $ denote  the  $H$-orbit of $v$. The map  $ V \to \bar V,  v \to \bar v,    v\in V  $  induces a bijection between   $V/G$  and  $\bar V/ (G/H)  $.  If the  action of  $H$   on $V$  is fixed point free then $ \Stab_{G/H}  (\bar v)   \cong  \Stab_G(v)  $ for all  $v \in V$.
\end{Lemma}

\begin{proof}    The first   assertion  is  clear.   Now suppose that  $H$ acts fixed point freely on  $V$ and let $v \in V$. Then the   canonical    homomorphism   $G \to G/H $  restricts to   an injective   homomorphism  from  $ \Stab_G(v)  \to \Stab_{G/H} (\bar v )  $. Let  $I $ be the  full  inverse  image of  $\Stab_{G/H} (\bar v) $. Then   $I$ acts   on $\bar v$  and   $ H$ acts transitively on $\bar v $.  Hence by the Frattini argument, $I = H \Stab_{I} (v) =  H\Stab_G(v)$  and  it follows that  the map   $ \Stab_G(v)  \to \Stab_{G/H} (\bar v )  $ is also surjective.
\end{proof}

\begin{Lemma}\label{weight:Q-S}  Let $\CF$ be  a saturated fusion system on  a finite   $p$-group $S$    and let $Q\leq   S$ be  a characteristic subgroup  of $S$ of index $p$  with $C_S(Q) =Z(Q)$.  Let $a $ be a positive integer with $p^a \leq |S|$ and let  $V$ be   the  subset of $\Irr^{a}(Q)$   consisting of   characters $\chi $   which are not   $S$-stable.  Suppose that $\Irr^{a-1}(Q)  =\emptyset $. 
Then  \[ \w_S(\CF, 0 , a)  =    \sum_{v \in  V/J }  z(k\Stab_J(v) ) \]  where $J=  N_{\Out_{\CF} (Q) } (\Out_S(Q) )$.
If in addition  $Q$ is  $\CF$-radical, then  $\w_S(\CF, 0 , a)  $  equals the  negative of  the contribution of  the   $\Out_{\CF}(Q)$-class of the  chain    $(1 < \Out_S(Q)) $  to $\w_{Q} (\CF, 0, a)   $. 
\end{Lemma}

\begin{proof}    Since $Q$  is  characteristic in  $S$,  $\Aut_Q(S)  $  is  a normal subgroup of $\Aut_{\CF} (S) $  and  conjugation  induces an action of     $ G:=\Aut_{\CF} (S) / \Aut_Q(S) $    on $\Irr^{a}(Q) $  via conjugation.   In particular,     $V$  is   $G$-stable and since $Q$ is of index $p$ in $S$, the normal subgroup $ H:=\Inn (S)/  \Aut_Q(S) $   of     $G$   acts fixed point freely  on $V$.   Let $\bar V =  V/H $ and note that $G/H \cong  \Out_{\CF} (S) $.

Let $\psi \in \Irr(S)$  and let $\chi \in \Irr (Q)$ be covered by $\psi $.   If $\psi  \in \Irr^{a}(S)$  then  since  $Q$ has no characters of defect $a-1 $,  $\chi$   has  defect $a$, $\chi $ is not $S$-stable and $\psi =\Ind_{Q}^S (\chi)$.  Conversely, if $\chi \in \Irr^a(Q) $ is not $S$-stable, then $\Ind_{Q}^{S} (\chi)$ is an  irreducible character of $S$ of defect $a$.    Since the set  of irreducible  characters of   $Q$ covered by    a given irreducible character of $S$ is  an $S$-orbit   under conjugation,   it follows that the map which sends  an  irreducible character  $\chi $  of $S$ to the  set   of irreducible characters  of $Q$  covered by  it induces a  $G/H  $-equivariant   bijection between $\Irr^{a} (S)$  and     $\bar  V $.  By  Lemma~\ref{lem:fffree},  there is a bijection between the set of $G$-orbits  of $V$ and  the set of 
$\Out_{\CF}(S) =G/H$-orbits of $\bar V$  which preserves the  isomorphism type of  the  point stabilisers.  Hence,
\[\w_{S} (\CF, 0, a)  =  \sum_{ v \in V/G}   z(k\Stab_{G} (v)).\]

By Lemma~\ref{lem:trivnontriv},  restriction to  $Q$ induces an isomorphism  $ G=\Aut_{\CF}(S)/\Aut_Q(S)  \cong J $   and this isomorphism is  equivariant with respect to the action of the  two groups on $\Irr^{a}(Q)$.  The   first assertion of the result follows  from  this and the  above  displayed equation.
The second  assertion follows from the  first   by Lemma~\ref{lem:max}.
\end{proof}

\section{Orbits and stabilisers  for   some  simple    $\GL_2(p)$-modules}

We begin  by recording an elementary fact.
 
\begin{Lemma} \label{lemma:perm} Let $G$  be a  group  acting on  a set $V$ and let $H \leq G$. Let  $A$ be the  subset of  $V$   consisting of all elements $v \in V$  such  that $\Stab_G(v) =  H$  and let  $B$ be the subset of $V$  consisting of  all elements $v$ in $V$  such that $\Stab_G(v)  $  is  a  $G$-conjugate of  $H$.   Then   the inclusion of $A$  in  $B$ induces a bijection between the set of $N_G(H)/H$-orbits  of  $A$ and the set of  $G$-orbits of $B$.  Consequently, the  number of $G$-orbits   of elements whose stabiliser  is a $G$-conjugate of  $H$ equals   $\frac{|A| |H|} {|N_G(H)|}$.
\end{Lemma}

For the  rest of the section  we assume  that   $p\geq 5 $. We let $G =\GL_2(p) $  and let  $ G_0=  \SL_2(p )$. Let $T$ be the  subgroup of diagonal matrices  of $G$ and $U$ the subgroup  of  lower unitriangular matrices of $G$.  If  $p\equiv  1 \pmod 3 $,  let   $T_1= \left\langle  \begin{pmatrix}  \omega   & 0 \\  0 & \omega^{-1} \end{pmatrix}  \right\rangle $, where $\omega \in \mathbb F_p^{\times} $ is a primitive third root of unity  and if $p\equiv  2 \pmod 3 $, let  $T_1= \left\langle \begin{pmatrix} 0 & -1\\1 & -1 \end{pmatrix} \right\rangle $, a subgroup of $G$  order $3$. Let $ t =    \begin{pmatrix}  0 & 1\\ 1 & 0 \end{pmatrix}$. 

\begin{Lemma}  \label{lem:3sylow} 
The following statements hold:
\begin{enumerate} [(i)] 
\item   $G_0$  has cyclic Sylow $3$-subgroups.
\item   Up to $G_0$-conjugacy,  $T_1 $  is the  unique   subgroup of order $3$  of $G_0$.
\item If $p\equiv 1 \pmod 3 $, then  $C_G(T_1) =T $ and $N_{G}(T_1)  =   T \rtimes \left\langle  t   \right\rangle $.
\item  If $p\equiv 2 \pmod 3 $, then  $C_G(T_1) \cong C_{p^2-1}$,  $N_{G}(T_1)  =    C_G(T_1)  \rtimes \left\langle  t   \right\rangle $ and   for any $g \in C_G(T_1) $, $ tgt^{-1} = g^p $. Moreover,  up to $G$-conjugacy,  $ T_1 \rtimes \left\langle t \right\rangle $  is  the unique subgroup of $G$  which is isomorphic to $S_3 $  and whose Sylow $3$-subgroup is contained in $G_0 $.
\item If $p\equiv 2 \pmod 3 $, then  $Z(G) $ is the  subgroup of order $p-1$ of $C_G(T_1)$,   $C_{G_0} (T_1) \cong C_{p+1}$,  $N_{G_0}(T_1)  =    C_{G_0} (T_1) \left\langle  vt   \right\rangle $, where  
$v \in C_G(T_1) $ is an element of  order $2(p+1)$.  In particular, $|N_{G_0}(T_1)|=  2 (p+1) $.
\end{enumerate} 
\end{Lemma}
 
\begin{proof}   Note that  the uniqueness assertion in (ii) is a consequence of  (i). Suppose first that $p\equiv 1 \pmod  3 $.  Then  the Sylow $3$-subgroup of  $T$  is a Sylow $3$-subgroup of $G$, the Sylow $3$-subgroup of $T \cap  G_0$ is a Sylow $3$-subgroup of $G_0$   and $ T\cap G_0 $ is cyclic.  The assertions (i)-(iii) are immediate.   Now suppose  $p\equiv 2 \pmod  3 $ and let $  h \in G $ be  of order $p^2-1 $, (so $\left\langle h  \right\rangle $ is  a Singer cycle  of   $G$).   The Sylow $3$-subgroup of $\left\langle h \right\rangle $ is  a Sylow $3$-subgroup of  $G$ and  in particular   $G$  and $G_0$ have  cyclic Sylow $3$-subgroups, proving (i).   Now (ii) follows from  (i) by observing that  $T_1$ is of order $3$.  In particular, without loss of  generality we may assume that  $T_1  \leq \left\langle h \right\rangle $.  For (iv),  note that   $T_1$  acts irreducibly on   ${\mathbb F}_p^2 $ hence  $C_G(T_1) =\langle h \rangle \cong C_{p^2-1}$. By inspection, $ t \in  N_G(T_1) \setminus C_G(T_1)  $, hence $ t$  normalises  (but does not centralise)  $ C_G(T_1)$. Since  $\Aut (C_3) \cong C_2 $,    and $t$ has order $2$, $N_G(T_1) =  C_G(T_1)  \rtimes  \left\langle  t   \right\rangle  $.   Now,   $h$ and $tht^{-1} \ne h $ have the same minimal polynomial on  ${\mathbb F}_p^2 $  and hence the same eigenvalues, say $\{\lambda, \lambda^p  \}$. Since $tht^{-1}$  is also a  power of $h$,  it follows that  $tht^{-1}  =  h^p  $  and  hence  $ tgt^{-1} = g^p $ for all  $g \in C_G(T_1) =\left\langle h \right\rangle $ (see \cite[Theorem 2.9]{singer}).  Since  $N_G( T_1\rtimes \left\langle  t   \right\rangle) \leq  N_G(T_1)$, by (ii)   in order to prove the last assertion of (iv)  it suffices to prove that if $ g  \in C_G(T_1)  $ is  such that $ gt$ is an involution, then $gt$ is   
$N_G(T_1)$-conjugate to   $t$.   Now $gt $  is an involution if  and only if $g^{1+p}=1 $   and on the other hand  $gt $ is  $N_G(T_1)$-conjugate to $t$  if and only if $g =x^{p-1}$  for some $x \in \left\langle h \right\rangle $.    This proves (iv)  since $ \left\langle h^{p-1} \right\rangle $  is the subgroup of  order  $p+1 $ of $\left\langle  h \right\rangle $.   The first assertion of (v) is immediate since $Z(G) \leq  C_G(T_1) $  has order $p-1$. The  restriction of the determinant map   to $\left\langle h \right\rangle  $ is surjective, giving the second assertion of  (v).  Let  $v \in C_G(T_1)$ be an element of  order $2 (p+1) $. Then $v^2 \in   C_{G_0}(T_1)  $ and  $ v\notin C_{G_0}(T_1)$,  hence   $\det (v)=-1 $. Since $t$ also has determinant $-1$, this  proves the  last assertion of (v).  \end{proof}

 Let   $V$  be  an absolutely  irreducible, faithful   ${\mathbb F}_pG$-module      with $\dim_{{\mathbb F}_p} V=4 $.  By \S 30 of  \cite{BN}    there   exists  an integer  $r$,  $ 0 \leq r \leq p-2$  such that identifying $V$  with   the   ${\mathbb F}_p$-span of   degree $3$ homogenous  polynomials  in  ${\mathbb F}_p [x,y]$ we have that for   $ g= \begin{pmatrix} a& b\\ c&d \end{pmatrix}   \in G $,  and  $v= x^i y^{3-i} $,  $ 0\leq i \leq 3 $,  
\[  vg   =    \det(g)^r  (ax + by)^i    (cx +dy)^{3-i}.\]  
In particular,  if   $ g = \begin{pmatrix} a & 0 \\0 & a \end{pmatrix} \in Z(G) $,  then  $ vg   =    a^{2r+3}   v  $.  Since $ V$ is  faithful, $ 3+2r  $   is relatively prime to  $p-1 $  and   $z=\begin{pmatrix} -1 & 0\\ 0 &-1 \end{pmatrix}  $  acts as $-1$  on $V$.

 For $ 0\leq i \leq 3 $, let $V_i$ be the  $ {\mathbb F}_p $-span of $x^{3-i} y^i$. Then\[V = V_0 \oplus  V_1 \oplus V_2 \oplus V_3 \]  is  a  direct sum decomposition of $V$ into non-trivial  ${\mathbb F}_p T$-modules.   Further,   $\Ker_T(V_i) $  consists  of  the  diagonal matrices  $\begin{pmatrix} \alpha& 0\\ 0&\beta \end{pmatrix} $  such that $\alpha^{3+r-i} \beta^{r+1} =1 $,   $ T=\Ker_T(V_i) \times Z(G) $  and in particular,   $\Ker_T(V_i) \cong   C_{p-1}$.     A  generator  of  $U$ acts  as  a single Jordan block on $V$ with  respect to  the flag  \[V \supset  V_2\oplus V_1 \oplus  V_0  \supset  V_1\oplus V_0  \supset V_0  \supset \{0\} \] and $V^{U} =  V_0 $.

 \begin{Proposition}\label{prop:GL2-orbits_general}    $G$ has $6$ orbits on $V$  with point stabilisers isomorphic  to  
  $C_p \rtimes C_{p-1}$,  $S_3$,  $C_3$, $C_{p-1}$, $C_2$ and  $G$ respectively.
\end{Proposition} 

\begin{proof}   By the subgroup structure of  $\SL_2(p)$ (see \cite[Theorem~6.17]{Su}),   if $H$ is a  subgroup of $G$ not containing $G_0$, then   either $H$ is  a $p'$-group or  $H$  has a normal Sylow $p$-subgroup (necessarily $G$-conjugate to  $U$).  We will use  this fact frequently.

Let $H=\Stab_G(x^3)  $ and let $\eta$ be a  generator of ${\mathbb F}_p^{\times}$.  Since $p-1> 3 $,  $\eta^3 \ne 1 $   and  consequently   $\begin{pmatrix} \eta & 0\\ 0 &\eta^{-1} \end{pmatrix}  \notin H$.    The above  gives that  $ U\leq H \leq N_G(U)  = UT $.  It follows that $H = U \Ker_T(V_0) \cong C_p   \rtimes C_{p-1}$.    Further, since  $U$ is  the  unique Sylow $p$-subgroup of  $H$,    $N_G(H)  \leq  N_G(U)=UT $  and  since $T = \Ker_T(V_0) \times Z(G) $,    $N_G(H) =  UT  $.     Since $V^U =V_0 \subseteq V^H $, we have  that  $\{  v \in V \,: \,  \Stab_G(v) = H \}  =\{ \alpha  x^3 \,: \, \alpha \in  {\mathbb F}_p^{\times}\} $.
By Lemma~\ref{lemma:perm}, there is a unique  $G$-orbit    on $V$  with stabiliser   $G$-conjugate to $H$  and  this orbit has size  $p^2-1 $.  Since $V^{U} =  V^H $   and $U$ is a Sylow $p$-subgroup of $G$ normal in $H$,  this $G$-orbit  is also equal to the set  of  all non-zero elements of $V$ whose stabiliser in $G$  has order divisible by $p$.

Let $\bar G =\GL_2(\bar {\mathbb F}_p) $,  $\bar T$ be the  subgroup of diagonal matrices of $\bar G$,  and let $\bar T_0 =  \bar T \cap \bar G_0$.  Let  $\bar V= \bar {\mathbb  F}_p \otimes_{{\mathbb F}_p}   V$.   The   $G$-action  on  $\bar V$ extends  to a  $\bar G$-action  via
  \[  vg   =    \det(g)^r  (ax + by)^i    (cx +dy)^{3-i} \]    for $ g= \begin{pmatrix} a& b\\ c&d \end{pmatrix}   \in \bar G $,  and  $v= x^i y^{3-i} $,  $ 0\leq i \leq 3 $.     For  each $i$, $0\leq  i \leq 3 $,  let  $\bar V_i   =  \bar {\mathbb  F}_p \otimes_{{\mathbb F}_p}   V_i$. Then $\bar V_i $ is  $\bar T_i$ -invariant  and  it is  a straightforward  check that 
 \[  \Ker_{\bar T_0}  (V_1) = \Ker_{\bar T_0}  (V_2)  =1 \]   and  
  \[\Ker_{\bar T_0}  (V_0) = \Ker_{\bar T_0}  (V_3)  =      \left\langle  \begin{pmatrix}  \omega   & 0 \\  0 & \omega^{-1} \end{pmatrix}  \right\rangle, \]   
 where $\omega  \in   \bar {\mathbb F}_p^{\times} $ is a primitive $3$-rd root of unity. 

Let $ 1\ne h \in G_0 $ be  a  $p'$-element   stabilising a non-zero element of $V$. Then $ h$ is $\bar G$-conjugate to an element of $\bar G$ of  the  form  $ \left\langle\begin{pmatrix}  \alpha   & 0 \\  0 & \alpha^{-1} \end{pmatrix}  \right\rangle $ and  it follows  by the above that  $h $ has order $3$.  Hence by Lemma~\ref{lem:3sylow},  if    $1\ne H_0 \leq  G_0 $    is a   $p'$-subgroup of $G$  stabilising a  non-zero element of $V$  then $ H_0   $ is $G_0$-conjugate to $T_1$.  Moreover,
\[\dim_{{\mathbb F}_p}   V^{T_1} =\dim_{\bar {\mathbb F}_p}   {\bar V}^{   \left\langle  \begin{pmatrix}  \omega   & 0 \\  0 & \omega^{-1} \end{pmatrix}  \right\rangle    } =2. \]

Suppose first that $p \equiv 1 \pmod 3 $.  As  above, we have that  $V^{T_1}  =V_0 \oplus V_3 $. Now let  $0\ne w \in V^{T_1}  =V_0\oplus V_3 $  be such  that  $T_1$ is a proper subgroup of $\Stab_{G_0}(w) $.  Then   $ \Stab_{G_0} (w) $   has  a normal Sylow $p$-subgroup  of order $p$   with   $T_1 $  a  normal $p'$-complement. Thus,
 $\Stab_{G_0} (w) = \,^g U T_1 $   for some  $ g\in G_0 $.   Now  
 \[  \,^g U T_1 \leq      N_{G_0} (\,^g U)=  \,^g (UT_0) =  \,^g   U  \,^g T_0, \]    hence $\,^{g^{-1}}  T_1  \leq  UT_0$. Since the  Sylow $3$-subgroups of $UT_0$ are cyclic,  it follows that  $\,^{g^{-1}} T_1 = \,^u T_1  $ for some $u \in U$   and hence  $g=(nu)^{-1}$ for some $n \in N_G(T_1) $.  Now,  
$N_{G_0}(T_1) =  T_0 \left\langle  s \right\rangle $,   where  $s= \begin{pmatrix} 0 & 1\\-1 & 0 \end{pmatrix}$  and $T_0$ normalises $U$.  Hence, either $\,^gU =U$ or $\,^gU =  \,^sU $  is the group of lower unitriangular  matrices.   By direct calculation,
$V^{\,^sU}  =  V^{\,^sU T_1}= V_3$.  Since  also  $V^{U}  = V^{UT_1} =  V_0$,   we have that 
   \[\{ v \in V  :  \Stab_{G_0} (v) =  T_1\}    =  V_0 \oplus  V_3   \setminus {V_0 \cup V_3}= \{ \alpha  x^3  + \beta y^3  \, :   \alpha, \beta \in {\mathbb F}_p^{\times} \} .\]
Since  $|N_{G_0}(T_1) |= | T_0 \langle  s \rangle | =    2 (p-1)$,   by Lemma~\ref{lemma:perm}  we have that  there are  $\frac{3}{2} (p-1) $  $G_0$-orbits  with stabilisers  $G_0$-conjugate to  $  T_1$.

Now suppose that $p \equiv  2 \pmod 3 $.   Then $|N_{G_0} (U) | =  |UT_0|=  p(p-1) $  is relatively prime to $3$.  So,  no element of order $3$ of  $G_0$  normalises  a  non-trivial  $p$-subgroup of $G_0$ which means that
 $T_1 = \Stab_{G_0}  (w)  $   for all  $ w \in   V^{T_1} \setminus \{0\}$.    Since $V^{T_1} $  is   $2$-dimensional,   $|\{ v \in V  :  \Stab_{G_0} (v) =  T_1\} | =  p^2 -1 $. On the other hand,   by Lemma~\ref{lem:3sylow},  $ |N_{G_0} (T_1)|    =  2(p+1)  $, so again by Lemma~\ref{lemma:perm}  we have that  there are  $\frac{3}{2} (p-1) $  $G_0$-orbits  with stabilsers  $G_0$-conjugate to  $ T_1$.

 Thus, in both cases  there are $\frac{3}{2} (p-1) $  $G_0$-orbits  with  element stabilsers  $G_0$-conjugate to  $ T_1$  and  such orbits account for   $ \frac{1}{2} ( p (p-1)^2 (p+1))  $ elements of  $V$.  As proved earlier there are  $p^2-1 $ points of $V$ whose stabilisers  contain  a Sylow  $p$-subgroup of $G_0$. Thus,  by an element count  there are  $\frac{p+1}{2} $ free  $G_0$-orbits  of $V$.
 
Next, we consider $G$-orbits. We have already  seen that there is one $G$-orbit  with stabiliser  containing  a Sylow $p$-subgroup  of $G$  and the  point stabilisers  of this orbit are isomorphic to $C_p \rtimes C_{p-1} $.
Now  suppose  that   $H   \leq  G  $ is    a  $p'$-group stabilising a    non-zero point, say $w$  of $V$.  By the above, either $H \cap  G_0  =1 $  or 
$H \cap G_0  $ is  $G_0$-conjugate to  $T_1$.

Suppose  that  $ H \cap G_0 $  is   equal  to   $T_1$. Then $T_1$ is a normal subgroup of $H$  and $ H/ T_1$  is cyclic of order dividing $p-1 $.  Consider first the case that $p \equiv 1 \pmod 3 $.    By Lemma~\ref{lem:3sylow},    $C_G(T_1) = T $  and  $N_G(T_1)  = T  \rtimes \left\langle  t \right\rangle $. By the above 
$ w  =\alpha  x^3  +   \beta y^3  $  for some  $ \alpha, \beta \in {\mathbb F}_p^{\times}  $. In particular, $\Stab_T (w) =   \Ker_T(V_0) \cap \Ker_T (V_3)   $.    Now,  let  $  g \in \Ker_T(V_0) \cap \Ker_T (V_3)   $. Then    $  g=\begin{pmatrix} a & 0\\0 & b \end{pmatrix}$  with $ a, b \in  {\mathbb F}_p^{\times}  $ such that $ a^{3+r} b^r =1  =  a^r b^{3+r}  $.      Hence  $a^3 =b^3 $ and  $g^3 \in Z(G)$ . Since  $Z(G)$  does not stabilise   any non-zero point of $V$, it follows that $g^3=1  $. Now  $T_1$ is the unique  subgroup  of order $3$ of $\Ker_T(V_0) $, thus $g\in T_1$. It follows   that  either  $\Stab_G(w)  =T_1 $
or   $\Stab_G(w)  \cong   S_3 $.
    
Now let   $ g=  \begin{pmatrix} a & 0\\0 & b \end{pmatrix}   \begin{pmatrix} 0 & 1\\1 & 0 \end{pmatrix} =  
\begin{pmatrix} 0 & a\\b &  0\end{pmatrix}  $  be  an element of  $N_G(T_1) \setminus  T $.   Let  $ w =  x^3 +\beta y^3 $   for some  $\beta \in {\mathbb F}_q^{\times}$.  Then   $ g \in \Stab_G(w)$ if and only if
\[ (-1) ^r \beta  a^r b^{3+r} =  1 =  (-1)^r  \beta^{-1}  a^{r+3} b^3.  \] 
Multiplying the  two equations gives  $ (ab) ^{3+2r} =1 $. Since   $(ab)^{p-1}=1  $ and  $3+2r   $ is relatively prime to  $p-1  $, we   obtain $ab =1 $. Substituting  in   $ (-1) ^r \beta  a^r b^{3+r} =  1   $ gives  $  b^3 = (-1)^r \beta $.    Choosing $\beta    $ such that    $(-1)^r  \beta = \eta  $, we see that there exists  $w_1 \in V$ with $\Stab_G(w_1)  =  T_1 $  and setting $\beta = (-1)^r $   shows that there  is  a  $w_2\in  V$ with $ T_1 \leq  \Stab_G(w)  \cong   S_3 $.     Consequently,   the $G$-orbit of   $w_1 $   is a  union of  $p-1$   $G_0$-orbits of  type $T_1 $ and  the $G$-orbit of   $w_2 $   is a  union of  $\frac{1}{2} (p-1)$   $G_0$-orbits of  type $T_1 $.  Since there are a  total  of   $\frac{3}{2} (p-1) $ $G_0$-orbits  of  type $T_1$,   this accounts for  all  $G$-orbits  whose  point stabiliser  in  $G_0$  is   of type  $T_1 $.

Now consider the case   $p \equiv 2 \pmod 3 $ and set $T'= C_G(T_1) $. We claim that $\Stab_{T'}(w) =T_1$. Indeed,  let $ g \in  \Stab_{T'} (w) $.  Then $g^{p-1}  \in \Stab_{T'} (w) \cap    G_0=  T_1 $  and by Lemma~\ref{lem:3sylow}, $Z(G) $ is the subgroup of order $p-1$ of  $T'$.  Since $3 $ and $p-1$ are relatively prime, we  have that  $ g \in T_1 Z(G)$.   Now  $T_1'$ stabilises     $w$   whereas no non-identity element  of $Z(G)$  does, hence  $g \in T_1$, proving the claim.     By the claim, either $\Stab_G(w) =T_1$ or  $ T_1 \leq \Stab_G(w) \cong  S_3 $.   Thus  by Lemma~\ref{lem:3sylow}, we may assume that either $\Stab_G(w) =T_1$ or  $\Stab_G(w) = H:= T_1 \rtimes \left\langle  t \right\rangle $.
By a direct check,  $V^{\left\langle  t \right\rangle } $   is   a $2$-dimensional space  with basis  $\{x^3 +y^3, x^2y+  xy^2 \}$   if $r $ is even and  with basis  $\{x^3 -y^3, x^2y-xy^2 \}$   if $r$ is  odd.    Since  $x^3 \pm y^3 \notin V^{T_1} $, we have that $V^{H}  $ is  at most $1$-dimensional.  Suppose that $V^{H} $ is $1$-dimensional.  Since $N_G(H) \leq  N_G(T_1) =T' \rtimes  \left\langle t \right\rangle  $, by Lemma~\ref{lem:3sylow}, we  obtain  that   $N_G(H) =  H \times  Z(G)$. Thus by Lemma~\ref{lemma:perm}, there  is one $G$-orbit   of $V$ with point stabiliser  of type  $H$. Since there are   $\frac{3}{2} (p-1) $ $G_0$-orbits  of  type $T_1$, these  must be the union of one $G$-orbit of type $H$  and one $G$-orbit of type $T_1 $.    If  $H$ fixes no  non-identity element of  $V$, then   every  $G_0$-orbit of type $T_1 $ lies in a $G$-orbit of type $T_1$. This is a contradiction as  every $G$-orbit of type $T_1$ is the union of     $(p-1)$  $G_0$-orbits  of type $T_1 $   and there are $\frac{3}{2} (p-1) $ $G_0$-orbits  of  type $T_1$.

Finally, consider  the $G$-orbits  on the  union of the  $\frac{p+1}{2} $ free $G_0$-orbits.  By direct calculation,    one can check that  
\[ \Stab_{G_0} (xy^2)  =    \Stab_{G_0} (x^3 +xy^2)   =1.  \]  
Thus $\Stab_G(xy^2)$  and  $\Stab_G(xy^2)$  are  both  cyclic of order dividing $p-1$.  But  $ \Stab_{T} (xy^2) =\Ker_T (V_2) \cong C_{p-1}$. It follows that  the  $G$-orbit  of $xy^2 $ is the $G_0$-orbit of $xy^2$.
 Let   \[T_2=\left\langle    \begin{pmatrix}  1 & 0 \\  0 & -1 \end{pmatrix}    \right\rangle,      T_3=\left\langle    \begin{pmatrix}  -1 & 0 \\  0 & 1 \end{pmatrix}    \right\rangle.  \]

We claim that $\Stab_G(x^3 + xy^2) \cong  C_2$. Indeed, suppose first  that $r$ is even. Then   $\Stab_{T} (x^3+xy^2) =  \Ker_{T}( V_0) \cap  \Ker_T(V_2) = T_2 $. On the other hand,   since  $ \Stab_G(x^3 + xy^2)  $ is cyclic,   $ \Stab_G(x^3 + xy^2)  \leq C_G(T_2)=T$. Thus,
$\Stab_G(x^3+xy^2) = T_2$. If  $r $  is odd, then by similar arguments  one shows  $ \Stab_G(x^3 + xy^2)  = T_3$, proving the   claim. It follows that  the  $G$-orbit of $(x^3+xy^2) $ is a union of $ \frac{p-1}{2}$ free $G_0$-orbits.
This accounts  for all   free $G_0$-orbits,  proving the proposition.
\end{proof}

 Now suppose $p=7 $. Set $I=UT$    and let     $\eta  $  be  a  generator of  $ {\mathbb F}_7^{\times}$  with   $\omega = \eta^2 $.  Let $T_1$  be  as  defined  earlier in this section and let  $T_2$, $T_3$   be  as in the proof of Prop.~\ref{prop:GL2-orbits_general}, that is 
     \[T_1    =\left\langle  \begin{pmatrix}  \omega  & 0 \\  0 & \omega^{-1} \end{pmatrix}  \right\rangle,    T_2=\left\langle    \begin{pmatrix}  1 & 0 \\  0 & -1 \end{pmatrix}    \right\rangle,      T_3=\left\langle    \begin{pmatrix}  -1 & 0 \\  0 & 1 \end{pmatrix}    \right\rangle.  \]
Define further subgroups $T_i $  and  $ I_i $ of $UT $, for $4 \leq  i \leq 6$,  as follows:
 \[T_4   =\left\langle  \begin{pmatrix}  \eta  & 0 \\  0 & 1 \end{pmatrix}  \right\rangle,    T_5=  T_4T_1,    T_6=  T_4T_2,      I_i =  UT_i.\]  
 
Recall as above that for any $i$, $0\leq i \leq  3$,  $ \Ker_T (V_i) $ is   a cyclic group of order $6$ consisting of  the  diagonal matrices  $\begin{pmatrix} \alpha& 0\\ 0&\beta \end{pmatrix} $,  $\alpha, \beta \in {\mathbb F}_7^{\times}   $  and   $\alpha^{3-i+r} \beta^{i+r}  = 1 $.

\iffalse
\begin{Lemma}\label{lem:T}     \begin{enumerate}
 \item[(i)]   $\Ker_T(V_i)  \cap  \Ker_T(V_{i+1} )=\{1\} $  for any $i$, $0\leq i < 3 $. Consequently, $  \Ker_T(V_i) \cap  \Ker_T(V_j) \cap \Ker_T(V_k) =\{1 \} $  for any $0\leq i<j<k \leq 3$.
 \item[(ii)] If $ r= 1, 5 $, then   $\Ker_T(V_0) \cap \Ker_T(V_2) = T_3$,  $ \Ker_T(V_1) \cap \Ker_T(V_3)=  T_2$,   $\Ker_T(V_0) \cap \Ker_T(V_3) =  T_1$.
 \item [(iii)] If $ r= 2,4$, then   $\Ker_T(V_0) \cap \Ker_T(V_2) =    T_2$,  $ \Ker_T(V_1) \cap \Ker_T(V_3)=   T_3$,   $\Ker_T(V_0) \cap \Ker_T(V_3) = T_1$.
 \end{enumerate}
\end{Lemma}
\fi 
 
\begin{Proposition} \label{prop:UTi-orbits}      Suppose that  $p=7 $ and $r=5 $.  Let   $J$ be  a subgroup of $I$ containing $U$. If $v \in V_0 $, then  $O_7(\Stab_J(v) ) \ne 1 $ and  if  $v\notin  V_0$, then   $\Stab_J(v)   $ is a  cyclic group of order  dividing $6$.  Set  $V' =  V \setminus V_0 $.  The  $I$-orbits of  $V'$  are  given as follows:
\begin{itemize}
\item[(i)]   $3$ orbits  with  point stabiliser of type $C_6$.
\item[(ii)]   $4$ orbits with  point  stabiliser  of type $C_2$.
\item[(iii)]  $3$ orbits with point stabiliser  of type  $C_3$.
\item[(iv)]   $6$ orbits  with  trivial  point stabiliser.
\end{itemize}
 The  $I_4$-orbits     are given as follows:
\begin{itemize}
\item[(v)]  $6$ orbits  with  point stabiliser of type $C_6$.
\item[(vi)]  $12$ orbits with  point  stabiliser  of type $C_2$.
\item[(vii)]  $50$ orbits  with  trivial  point stabiliser.
\end{itemize}
The  $I_5$-orbits     are given as follows:
\begin{itemize}
\item[(viii)]  $2$ orbits  with  point stabiliser of type $C_6$.
\item[(ix)] $8$ orbits  with  point stabiliser of  type $C_3$.
\item[(x)] $4$-orbits  with  point stabiliser of  type $C_2$.
\item[(xi)] $12$ orbits with trivial point stabiliser.
\end{itemize}
The  $I_6$-orbits     are given as follows:
\begin{itemize}
\item[(xii)] $3$ orbits  with  point stabiliser of  type $C_6$.
\item[(xiii)] $14$ orbits  with  point stabiliser of type $C_2$.
\item[(xiv)] $21$ orbits  with trivial point stabiliser.
\end{itemize}
\end{Proposition}

\begin{proof}  Since   $V^{U}=V_0 $  and   $U$ is   normal in $J$, we  have the  first assertion.  Since $U$ is  the unique  non-trivial $7$-subgroup of $J$,  it also follows that   if $v \in V' $, then  $\Stab_J(v)  $ is a $7'$-group.   By the   Schur--Zassenhaus theorem, $\Stab_J(v)  $ is  $J$-conjugate to a subgroup of $J \cap  T$.  On the other hand,  $\Stab_J(v) \cap Z(G) =1$ and $T/Z(G) \cong C_6$,   whence the second assertion.  Hence in order to  find  $J$-orbits on  $V'$,  it suffices  by Lemma~\ref{lemma:perm} to   determine for each subgroup  $X$ of  $J\cap T_0$  the set $ A_X:= \{ v \in V':  \Stab_T(v) = X\} $.  We note also  that no two distinct subgroups of $T$  are  $J$-conjugate and   hence distinct subgroups of $J\cap T$ will  give rise to   $J$-orbits   with  distinct stabiliser classes.

Now let $1\ne X$ be a subgroup of $T$.  A straightforward calculation gives that    $\Ker_T(V_0) \cap \Ker_T(V_2) = T_3$,  $ \Ker_T(V_1) \cap \Ker_T(V_3)=  T_2$,   $\Ker_T(V_0) \cap \Ker_T(V_3) =  T_1$  and  we obtain the following description  of   $A_X $:  
\begin{itemize}
\item[(i)]  $ V_0 \oplus  V_3 \setminus V_0 \cup  V_3$ if $X = T_1 \cong C_3 $,   
\item[(ii)]  $ V_1\oplus  V_3 \setminus   V_1\cup  V_3 $ if $X = T_2 \cong C_2 $, 
\item[(iii)] $ V_0\oplus  V_2 \setminus V_0 \cup V_2$ if $X = T_3 \cong C_2 $,
\item[(iv)]  $V_1 \setminus \{0\}$  if $X=  \left\langle \begin{pmatrix} 1 & 0\\  0 &\eta \end{pmatrix}  \right\rangle  \cong  C_6 $,
\item[(v)] $  V_2 \setminus   \{0\}$  if $X=  T_4 \cong  C_6 $,
\item[(vi)] $ V_3 \setminus \{0\} $ if $X=  \left\langle \begin{pmatrix} \eta^2 & 0\\  0 &\eta \end{pmatrix}  \right\rangle \cong  C_6 $,
\item[(vii)] $A_X =\emptyset  $ for all other  $X$.
\end{itemize}  
 
Now   observing that  $N_I(X)  = T$ for  any     subgroup  of $T$ which intersects $Z(G)$ trivially, Lemma~\ref{lemma:perm}   gives that  non-free $I$-orbits  of $V'$   are as asserted. The  number of free $I$-orbits   follows by an   element   count.

Let $i=4 $  and let  $1\ne  X $ be a  subgroup of  $I_4 \cap T$.  Since  $T_1 \cap  T_4 =  1 $,  $T_2 \cap T_4 = 1 $, $ T_3 \cap T_4 =  T_3 $,   $ \left\langle \begin{pmatrix} 1 & 0\\  0 &\eta \end{pmatrix}\right\rangle \cap T_4 =1 $,  and  $  \left\langle \begin{pmatrix} \eta^2 & 0\\  0 &\eta \end{pmatrix} \right\rangle  \cap T_4 =1$, we  have  that $A_X= V_0\oplus  V_2 \setminus V_0 \cup V_2$ if $X = T_3 \cong C_2 $, $A_X=  V_2 \setminus   \{0\}$  if $X=  T_4 \cong  C_6 $  and  $A_X =\emptyset  $ for all  other $X \ne 1 $. This gives the result for $i=4 $.

Let $i=5$ and let  $1\ne X $ be a  subgroup of  $I_5 \cap T$. Since  $T_1\cap T_5=  T_1$,   $T_2 \cap T_5 = 1 $, $ T_3 \cap T_5 =  T_3 $,   $ \left\langle \begin{pmatrix} 1 & 0\\  0 &\eta \end{pmatrix}\right\rangle \cap T_5 =   \left\langle \begin{pmatrix} 1 & 0\\  0 &\omega \end{pmatrix}\right\rangle  $,  $T_4 \cap T_5 = T_4 $, and  $  \left\langle \begin{pmatrix} \eta^2 & 0\\  0 &\eta \end{pmatrix} \right\rangle  \cap T_5 =T_1$, we  have  that 
$A_X= V_0\oplus  V_3 \setminus V_0 $ if $X = T_1 \cong C_3 $,  $A_X= V_0\oplus  V_2 \setminus V_0 \cup V_2$ if $X = T_3 \cong C_2 $,  $A_X=  V_2 \setminus   \{0\}$  if $X=  T_4 \cong  C_6 $  and  $A_X =\emptyset  $ for all  other $X \ne 1 $. 
This gives the result for $i=5 $.

Let  $i=6 $ and let  $1\ne X $ be a  subgroup of  $I_6\cap T$. Since $T_1\cap  T_6 =1 $,  $T_2 \cap T_6 =  T_2$, $T_3 \cap T_6 =T_3 $, 
 $ \left\langle \begin{pmatrix} 1 & 0\\  0 &\eta \end{pmatrix}\right\rangle \cap T_6 =T_2 $,  $T_4\cap T_6 = T_4 $ and $  \left\langle \begin{pmatrix} \eta^2 & 0\\  0 &\eta \end{pmatrix} \right\rangle  \cap T_4 =T_2$,  we have that  $A_X =  V_1\oplus V_3 \setminus \{0\} $  if $X= T_2$,   $A_X=  V_0 \oplus V_2\setminus V_0 \cup V_2 $ if $X=T_3$, $A_X=   V_2\setminus \{0\}$ if $X=T_4 $ and $A_X =\emptyset  $ for all  other $X \ne 1 $.  This  gives the result for $i=6$.
\end{proof}

\section{Calculating weights}
Throughout this section   we assume that $p\geq 5 $. We  let $S$  be  a Sylow $p$-group of $G_2(p)$   and let  $\CF$  be  a  Parker-Semeraro   fusion system on $S$, that  is,  a  saturated fusion system on  $S$ with $O_p(\CF)=1 $.  In what follows we freely use the notation  of  \cite{PS18}. In particular,  recall  from Sections 3 and 4  of \cite{PS18}   that  $S$ is  a  group of order  $p^6$,   nilpotency class   $5$,    with  a  distinguished set of generators  $x_i $, $ 1\leq i \leq 6 $,  such that  setting  $Q=  \left\langle x_2, x_3,  x_4,  x_5, x_6 \right\rangle  $  and  $R= \left\langle x_1,  x_3, x_4, x_5, x_6 \right\rangle $,     $Q$ and $R$ are maximal and  characteristic in $S$.    Moreover,  $S$ is  of exponent $p$ if $p\ne 5 $ and of exponent $5^2$  if $p=5 $.  Set $Z= Z(S) $, for $x \in S \setminus (Q \cup R)$, set $W_x = \left\langle Z, x \right\rangle  $ and set $\CW   =\{ W_x \mid  x \in S \setminus (Q \cup R)\} $.   If $p\geq  7$,  then the  elements of $\CW$ are   elementary  abelian of order  $p^2$. Let $ Z_i= Z_i(S) $ be the $i$-th  term  of the  upper central series of  $S$.

\iffalse
For the  rest of this section   let  $p=7 $,  let $S$  be  a  Sylow $7$-subgroup   of $G_2(7)$  and let $\CF$ be  a saturated fusion system on  $S$.   We freely use notation  and facts from  \cite{PS18}.   In particular, we recall that  $S$ is  a  group of order  $7^6$,   exponent $7$  and nilpotency class   $5$  with  a  distinguished set of generators  $x_i $, $ 1\leq i \leq 6 $  such that  setting  $Q=  \left\langle x_2, x_3,  x_4,  x_5, x_6 \right\rangle  $  and  $R= \left\langle x_1,  x_3, x_4, x_5, x_6 \right\rangle $,     $Q$ and $R$ are maximal and  characteristic in $S$.    Set $Z= Z(S) $, for $x \in S \setminus (Q \cup R)$, set $W_x = \left\langle Z, x \right\rangle  $ and set $\CW   =\{ W_x \mid  x \in S \setminus (Q \cup R)\} $.   The elements of $\CW$ are  elementary abelian   groups of order $7^2$.
 \fi 
  
The   classification  in  \cite{PS18} proceeds via identification of  essential subgroups whereas   for our purposes we  need  to identify the {\it a priori}  larger  class of radical-centric subgroups.   It turns out that   if $p\geq 7 $, then  every  proper  $\CF$-centric-radical  subgroup is $\CF$-conjugate to   an $\CF$-essential    subgroup.    This is implicit in the   analysis  in \cite{PS18}  and \cite{vB} - we  supply a few  details.

\begin{Lemma}  \label{lem:max subgroups}  Let $V$ be a $3$-dimensional  vector space over ${\mathbb F}_p $ and let  $G \leq  \GL(V)$ with  $O_p(G)=1 $.  Then   $|G|_p   \ne p^2$  and   if $|G|_p = p^3 $, then  $G$  contains $\SL_3(p) $. \end{Lemma}  
\begin{proof}   This follows by an inspection of the list of  maximal subgroups of  $\mathrm{PSL}_3(p) $ (see  for instance \cite[Theorem~6.5.3]{GLS3}).  \end{proof} 

The following is  an  analogue of  Lemma~4.3 of \cite{PS18}. For this,  note that   the conclusion of Lemma 2.9 of \cite{PS18}  can be strengthened to  the assertion that $E$  is not $\CF$-centric and $\CF$-radical (this  can be seen  by following the proof or  noting that this  is  a  consequence  of \cite[Proposition~3.2 and Lemma~3.4]{OV}).

\begin{Lemma} \label{lem:PS4.3}   If  $E\leq Q$ is $\CF$-centric-radical, then  either $E=Q$ or $E$ is $\CF$-conjugate to $Z_3 (S)$.
\end{Lemma} 

\begin{proof}  Suppose first that $E$ is non-abelian.   Then $E'  =Q'$  since   $Q'$ is   of order $p$.  Since $E/Q' \leq Q/Q'$ is elementary abelian it follows that $\Phi (E)  = E' = Q'= \Phi(Q) =Z(Q) $. Further, $[E, N_Q(E)]  \leq  \Phi(E) $ and $[\Phi (Q), N_Q(E)]=1 $. Then by \cite[Lemma~2.9]{PS18},  $E$ is not $\CF$-centric-radical.
So we may assume that $E\leq Q$  is    abelian.  Since $C_S(E) \leq E$,  $E$ is a   maximal abelian subgroup  of $Q$,  hence $E$  has order $p^3$. Also  $E$ contains  $Z(Q) =Q' $   which implies  that $E$ is normal in $Q$ (since all subgroups of the abelian group $Q/Q'$ are normal in  $Q/Q'$). Then $Q/E \leq  \Aut_{\CF} (E) \leq \GL_3(p)  $.   By assumption   $O_p(\Aut_{\CF} (E))=1 $  and $|Q/E| =p^2 $. Thus, by  Lemma~\ref{lem:max subgroups},  $\Aut_{\CF} (E) $ has Sylow $p$-subgroups of order $p^3 $.  This  means that  some $\CF$-conjugate, say $F$, of $E$ is normal in  $S$. By Lemma~3.2 (e)  of \cite{PS18},  $F=   Z_3  (S)$.
\end{proof}

\begin{Proposition} 
\label{martin}
 Let   $ X \leq S$ be  a proper, $\CF$-centric-radical subgroup. Then either $X$ is  $\CF$-conjugate to  an $\CF$-essential subgroup of $\CF$  or else $p=5$ and $X$ is $\CF$-conjugate to  $Z_3(S)$.  In particular,  if  either $Q$ or $R$ is $\CF$-centric-radical, then it is $\CF$-essential.
\end{Proposition}  

\begin{proof}  Suppose that  $X$ is  not  $\CF$-essential.  Then there exists an $\CF$-essential subgroup $Y$containing $X$ as a proper subgroup  (otherwise every  $\CF$-automorphism of $X$ would extend to an $\CF$-automorphism of  $S$ and $\Out_S(X) $   would be a normal non-trivial  $p$-subgroup of $\Out_{\CF}(X) $).  By  \cite[Theorem~4.2]{PS18}, either  $Y$ is one of $Q$ or  $R$, or  $p=7 $ and $Y$  belongs to the  family  $\CW$.  Suppose first that $Y \in \CW$.  Then  since $|X| > 7$,  and $|Y|=7^2 $, $ X=Y$,  a contradiction.  

Suppose next that  $Y$  is either $Q$ or $R$. Then $Y$ is characteristic  in $S$.   Suppose that only one of $Q$ or $R$ is $\CF$-essential or that  $X$ is contained in only one of  $Q$ and  $R$.  Then $Y$ is  the only $\CF$-essential  subgroup of $S$ containing $X$   and  since $Y$ is characteristic in $S$ it follows by Alperin's fusion theorem that  every $\CF$-automorphism of   $X$  extends  to an $\CF$-automorphism of   $Y$  and  consequently    $\Out_Y(X)$ is a normal non-trivial $p$-subgroup of   $\Out_{\CF}(X)$,  again a contradiction. Thus, both $Q$ and $R$ are $\CF$-essential and $X \leq Q \cap R$.  By Lemma~\ref{lem:PS4.3}, $X$ is  $\CF$-conjugate to $ Z_3(S)$ and in particular $Z_3(S)$ is $\CF$-centric and $\CF$-radical.   Thus   $N_{\CF}(Z_3(S))  $ is  a    saturated fusion system on  $S$  with 
$O_p(N_{\CF}(Z_3(S)) ) =  Z_3(S) $.  If $p \ge 7$, then by \cite[Thm~7.13]{vB}   there is no such  fusion system.   Finally,  if $X$ is one of $Q$ or $R$, then by maximality, $Y=X$.

\end{proof}

When $p=5$, of the five Parker--Semeraro fusion systems $\CF$ which occur, $X=Z_3(S)$ does appear as an $\CF$-centric radical subgroup in those of $\Ly$ and $\BM$. Indeed, we have $\Out_\CF(X) \cong \SL_3(5)$ in these two cases.

By  the above  and  \cite[Theorem 4.2] {PS18}, in order to calculate  $\CF$-weights, it suffices to  go through the  relevant elements of  $\{Q,  R, S, Z_3(S)\}  \cup {\CW} $. We will do this  in turn but first we gather   a  few additional   facts   from \cite{PS18}  and set up notation which will  be in effect for the rest of this section. Set  ${\mathbb F}={\mathbb F}_p $.

Let $L= {\mathbb F}^{\times }  \times \GL_2({\mathbb F}) $.  Define subgroups  \[K =\biggl\{  \biggl(\mu^{-3},  \begin{pmatrix}  \mu & 0 \\0  & \mu  \end{pmatrix} \biggr) \mid   \mu \in {\mathbb F}^{\times}\biggr\}, B_0  ={\mathbb F}^{\times }    \times \bigg\{    \begin{pmatrix}  \alpha & 0\\ \gamma  & \beta \end{pmatrix} \mid  \alpha,  \beta \in {\mathbb  F}^{\times}, \gamma  \in {\mathbb F} \biggr\},  S_0= \{ 1 \}    \times \bigg\{    \begin{pmatrix}  1& 0\\ \gamma  & 1 \end{pmatrix} \mid  \gamma  \in {\mathbb F} \biggr\} \]
\[   D=  \bigg\{  \biggl(t,   \begin{pmatrix} \alpha & 0\\0& 1 \end{pmatrix} \biggr)\mid t, \alpha  \in {\mathbb F}^{\times} \bigg\},   D'=  \bigg\{  \biggl(\alpha,   \begin{pmatrix} \alpha & 0\\0& 1 \end{pmatrix} \biggr)\mid  \alpha  \in {\mathbb F}^{\times} \bigg\}.\] 
There  is a  right action  of   $L$   on  $Q$  with kernel  $K$  and  setting  $ B=B_0   Q$,   we have   $ S =S_0Q$  and      $x_1 = \begin{pmatrix} 1 & 0  \\ 1 & 1 \end{pmatrix} $.     Note that  $ C_B(S) =  K  Z(S) $,  and  $ B =  K  D  S$.

The following will be in use throughout (see Lemma 4.6 of \cite{PS18}):

\

{\bf Assumption.}  $\Aut_{\CF} (S) \leq \Aut_{B}(S)  $.

\iffalse
\begin{Proposition} \label{prop:descF_general}   
Suppose $p \ge 5$ and let    $B'$    be  the full  inverse image  in $B$ of  $\Aut_{\CF}(S) $.   If $\Out_\CF(S) \cong C_{p-1} \times C_{p-1}$ then $D \le B'$. If  $ P\leq  S $  is    fully $\CF$-centralised, then $N_{\Aut_{\CF} (P)} (\Aut_S (P) ) =\Aut_{B'}(P) $  and    $N_{\Out_{\CF} (P)} (\Out_S (P) ) =\Aut_{B'}(P) $. 
\end{Proposition}

\begin{proof}    By construction,  $ K \leq C_B(S) $  hence $SK \leq B' $. Since $B = KD S $  it follows that $ B= K (B'\cap D)   S$.    Lemmas  4.12 and 5.5 of \cite{PS18} give that $D'\leq D $.      Now $\Out_{\CF} (S)  =  \Out_{B'\cap  D} (S)$, $ D \cong  C_{p-1} \times C_{p-1} $ and  $D'\cong C_{p-1}$ so $D \le B'$ if $\Out_\CF(S) \cong C_{p-1} \times C_{p-1}$.

Now suppose  that  $ P\leq  S $  is    fully $\CF$-centralised.   By the extension axiom of saturation,   every element  of  $N_{\Aut_{\CF} (S)} (\Aut_S (P) ) $ is the restriction to $P$ of  an element of 
$ \Aut_{\CF}(S)  $. Thus, $N_{\Aut_{\CF} (S)} (\Aut_S (P) ) \leq \Aut_{B'} (P)$.    Since the  restriction of any $\CF$ morphism  to a subgroup is  an $\CF$-morphism,  $\Aut_{B'} (S) = \Aut_{\CF}(S) $  and $S$ is  normal in $B$, we also have $\Aut_{B'}(P) \leq  N_{\Aut_{\CF} (S)} (\Aut_S (P) )$.   Thus $N_{\Aut_{\CF} (P)} (\Aut_S (P) ) =\Aut_{B'}(P) $  and consequently $N_{\Out_{\CF} (P)} (\Out_S (P) ) =\Out_{B'}(P) $ 
\end{proof}  
\fi

\begin{Proposition} \label{prop:descF_general}   Let    $B'$    be  the full  inverse image  in $B$ of  $\Aut_{\CF}(S) $.     Then $ B' =  K (B'\cap D)   S$,  $D'\leq B'\cap D $   and  one  of the following holds.
\begin{enumerate} [\rm(i)]
 \item  $\Out_{\CF}(S)  \cong C_{p-1} \times C_{p-1}$ and $ D \leq  B' $.
\item   $p=7$,   $\Out_{\CF}(S)  \cong C_6 $  and $D\cap  B' = D' $.
\item  $p=7$,    $\Out_{\CF}(S)  \cong C_6 \times C_2 $  and $(D\cap  B')/ D'  \cong C_2$.  In this case, \[   B'\cap D=  \bigg\{  \biggl(\pm \alpha,   \begin{pmatrix} \alpha & 0\\0& 1 \end{pmatrix} \biggr),  \alpha  \in {\mathbb F}^{\times} \bigg\}.\] 
 \item  $p=7$,  $\Out_{\CF}(S)  \cong C_6 \times C_3$  and   $(D\cap  B')/ D'  \cong C_3$.  In this case, letting  $\omega \in  {\mathbb F}^{\times} $ be a  primitive $3$-rd root of  unity,
  \[   B'\cap D=  \bigg\{  \biggl( \omega^k\alpha,   \begin{pmatrix} \alpha & 0\\0& 1 \end{pmatrix} \biggr), k\in \{0, 1, 2\}, \alpha  \in {\mathbb F}^{\times} \bigg\}.\] 
\end{enumerate}
If  $ P\leq  S $  is    fully $\CF$-centralised, then $N_{\Aut_{\CF} (P)} (\Aut_S (P) ) =\Aut_{B'}(P) $  and    $N_{\Out_{\CF} (P)} (\Out_S (P) ) =\Out_{B'}(P) $.   
\end{Proposition}

\begin{proof}    By construction,  $ K \leq C_B(S) $  hence $SK \leq B' $. Since $B = KD S $  it follows that $ B'= K (B'\cap D)   S$.     Lemmas  4.12 and 5.5 of \cite{PS18} give that $D'\leq B' $. Moreover,  Lemma 4.2, Lemma 5.5 and Theorem 5.9 of \cite{PS18} give that if $p\neq 7$ then only case (i) is possible.     Now suppose $p=7$, then $\Out_{\CF} (S)  =  \Out_{B'\cap  D} (S) \cong \Aut_{B'\cap D}(S) $, $ D \cong  C_6 \times C_6 $ and  $D'\cong C_6  $  hence  $\Out_{\CF}(S) $   is  isomorphic to  one of the groups  $C_6\times C_6 $, $C_6 $, $C_6 \times  C_2 $ or $C_6 \times C_3 $  giving rise  to the cases (i)-(iv).   The description of $B'\cap D  $ in cases (iii) and (iv)   follows from the fact that $D/D'$ has unique subgroups of order $2$ and $3$.   

Now suppose  that  $ P\leq  S $  is    fully $\CF$-centralised.   By the extension axiom of saturation,   every element  of  $N_{\Aut_{\CF} (S)} (\Aut_S (P) ) $ is the restriction to $P$ of  an element of 
$ \Aut_{\CF}(S)  $. Thus, $N_{\Aut_{\CF} (S)} (\Aut_S (P) ) \leq \Aut_{B'} (P)$.    Since the  restriction of any $\CF$-morphism  to a subgroup is  an $\CF$-morphism,  $\Aut_{B'} (S) = \Aut_{\CF}(S) $  and $S$ is  normal in $B$, we also have $\Aut_{B'}(P) \leq  N_{\Aut_{\CF} (S)} (\Aut_S (P) )$.   Thus $N_{\Aut_{\CF} (P)} (\Aut_S (P) ) =\Aut_{B'}(P) $  and consequently $N_{\Out_{\CF} (P)} (\Out_S (P) ) =\Out_{B'}(P) $ 
\end{proof}

As in  Section~2   let $k$  be  an algebraically closed  field of characteristic  $p$.
\begin{Proposition} \label{prop:2dim}  Let   $G$  be a subgroup of  $ \GL_2(p) $    containing $\SL_2(p)  $. Let $U \leq  \SL_2(p) $ be the subgroup of  lower unitriangular  matrices  and let $I =N_G(U) $.  Let $V$ be  a  $2$-dimensional   simple  ${\mathbb F} G$-module.   Then \[ \sum_{ v  \in V /G } z(k(\Stab_G(v) ) =  \sum_{v\in V/I}  z(k(\Stab_I(v) ). \]
\end{Proposition} 

\begin{proof}    The  restriction  of   $V $   to   $\SL_2(p) $ is the  natural   $2$-dimensional module.     Now $\SL_2(p)$  acts transitively on $V \setminus\{0\} $ and  for any $v \in V \setminus \{0\} $,  $O_p(\Stab_{\SL_2(p)}(v)) \ne 1 $, hence  $O_p(\Stab_G(v)) \ne 1 $   from which it follows that   $z(k(\Stab_G(v) ) =0$.  On the other hand,  $\Stab_G(0) =  G $     and  the  projective simple  $kG$-modules are exactly the extensions to $G$   of  the  Steinberg module   for $\SL_2(p) $, hence  $z(kG) = |G: \SL_2(p) | $.     It follows that the left hand side of  the  equation in the statement equals  $|G: \SL_2(p) | $.
Now  $ I \cap \SL_2(p) =N_{\SL_2(p)} (U)  =  U  T_0 $ where $T_0$ is the subgroup of diagonal matrices in  $\SL_2(p) $.     So,  if $ v=(a, 0), $    $a\in  {\mathbb F}^{\times}$, then  $U\leq \Stab_{I } (v) $ and hence  $z(k\Stab_{I}(v)) = 0 $. Let $X=\{ (a, b) \mid a, b \in {\mathbb F}, b \ne 0 \} $ and let $w \in X$.   The group  $UT_0 $  acts  simply  transitively on $X$, hence  $I $ acts transitively on  $X$ and  $|\Stab_I(w) | =  \frac{|I|} {|X |} =|G:\SL_2(p) | $.  Further,   $\Stab_{I}(w) \cap   \SL_2(p)  =\Stab_{UT_0}(w)=1 $, hence   $\Stab_I(w) $ is isomorphic to a subgroup of $G/ SL_2(p)  $   which  is a   cyclic  group.  It  follows that   the   right hand  side of the equation equals   $z(k\Stab_I(w))  =    |G:\SL_2(p) | $.  
\end{proof}

\subsection{The  family  $\CW$. }
\begin{Proposition}  \label{prop:W} Suppose that $W  \in  \CW$ is $\CF$-centric-radical.  Then $p=7 $,  
$\Out_{\CF}(W) \cong \SL_2(7)$ and  $\mathbf w_W(\CF, 0, d)=0$ for any $d$.
\end{Proposition}

\begin{proof}
The first two  assertions follow  from  Proposition~\ref{martin} and \cite[Lemma~4.12]{PS18}.   Since $W $ is abelian, all irreducible characters of $W$ are of degree $1$. Thus  $\Irr(W) = \Irr^2(W)$ and  $\mathbf w_W(\CF, 0, d)=0$ if $d \neq 2$.    There are two chains up to $\Out_{\CF}(W)$-conjugation; namely $\sigma_1=(1)$ and $\sigma_2=(1 < U)$ where $U$ is the subgroup of upper unitriangular matrices. Note that $I(\sigma_1)=\Out_{\CF}(W)\cong SL_2(7)$ and 
$I(\sigma_2)= N_{\SL_2(7)}(U) $.     Further,  since $ W  \cong C_7 \times C_7 $,     $\Irr(W)$   may be identified  with  a simple two dimensional $\SL_2(p)$-module. Now the result follows from Proposition~\ref{prop:2dim}.
%By \cite[Lemma 5.10]{PS18}, there are $\frac{36}{|\Out_{\mathcal F}(S)|}$ $\mathcal F$-conjugacy classes of $W$. 
\end{proof}

\subsection{Results for $Q$}

\begin{Proposition}\label{weightw_general}  $cd(Q) =\{1, p^2\} $ and  $\Irr(Q)$ consists of $p^4$ linear characters and 
$(p-1)$ faithful characters of degree $p^2$. Moreover, for any 
$\chi \in \Irr^3(Q)$, $\chi(x)=0$ for all $x\in Q \backslash Z(Q)$. \end{Proposition}

\begin{proof}
By Lemma 3.1 of \cite{PS18}, $Q$ is extraspecial of order $p^5$ and exponent 
$p$. Since $Q$ is extraspecial, we have that $Z(Q)=[Q,Q]$, so $Q/[Q,Q]$ has $p^4$ elements. Hence, there are $p^4$ 
linear characters of $Q$, in other words $|\Irr^5(Q)|=p^4$. On the other hand, Theorem 5.5.5 of \cite{G} implies that the 
faithful irreducible characters of $Q$ have degree $p^2$. Moreover, since $Q$ is of nilpotency class $2$, Proposition 
3.1 of \cite{BKL} implies that for any faithful character $\chi$ of $Z(Q)$, the map $\phi \mapsto \tau_\phi$ is a bijection 
between $\Irr(Z(Q)|\chi)$ and $\Irr(Q|\chi)$. Since there are exactly $(p-1)$ faithful irreducible characters of $Z(Q)$, we 
conclude that there are exactly $p-1$ irreducible characters of degree $p^2$ and
$$\Irr^3(Q)=\{\tau_\phi \ | \ \phi \in \Irr(Z(Q)) \text{ and $\phi$ is faithful}\}.$$
By the proof of Proposition 3.1 of \cite{BKL}, we have that $\tau_\phi(x)=0$ for all $x\in Q \backslash Z(Q)$. 
\end{proof}

We  record  the  possible  structure of $\Out_{\CF}(Q)$  when $ Q  $  is $\CF$-centric-radical   for $p\geq 7 $.

\begin{Lemma}  \label{lem:outQ} Suppose that  $p \geq 7 $ and $ Q$ is   $\CF$-centric-radical.    If $ p  =7  $, then either $ \Out_{\CF} (Q) \cong \GL_2(7) $  or $\Out_{\CF}(Q) \cong  C_3 \times 2.S_7 $  and if $p> 7 $, then  $ \Out_{\CF} (Q) \cong \GL_2(p) $. In all cases,    $Q/[Q,Q]  \cong  \mathbb F^4 $  is  an absolutely simple  and faithful  ${\mathbb F}  \Out_{\CF}(Q)$-module   via  the natural  action  of $\Out_{\CF} (Q) $  on  $Q/[Q,Q] $.\end{Lemma}

\begin{proof} By Proposition~\ref{martin}, $Q$ is  $\CF$-essential.   If  some element of  $\CW$ is  $\CF$-essential, then the hypothesis  of Lemma   5.2 of \cite{PS18} holds by \cite[Lemma~4.12]{PS18}.  Otherwise, by  \cite[Lemma 5.12]{PS18}, $R$ is $\CF$-essential and consequently  by   \cite[Lemma~5.3]{PS18}  the hypothesis  of Lemma   5.2  of \cite{PS18}  holds also in this case.
The  result now follows by  Lemma  5.2 of \cite{PS18}   and its proof.
\end{proof} 

\begin{Proposition}\label{weightq}
Suppose that $p\geq 7 $, $Q$ is $\CF$-centric-radical  and  $\Out_{\CF}(Q) \cong GL_2(p)$. Then   $ \mathbf w_Q(\CF, 0, 3)=1$.
\end{Proposition}

\begin{proof}
By Lemma~\ref{lem:max}, $\CN_Q  $ has two conjugacy classes  of chains, the trivial chain $\sigma_1  $ and the non-trivial  chain $\sigma_2:  (1\leq \Out_{S}(Q) )$.  From Theorem 1 of \cite{W}, we have that $\Aut(Q)=C_{\Aut(Q)}(Z(Q)) \cdot \left\langle \theta \right\rangle$, where $\theta$ 
is an automorphism of order $p-1$ which does not act trivially on $Z(Q)$. Note that 
$\Inn(Q)\leq C_{\Aut(Q)}(Z(Q))$ which yields that $\Out(Q)\cong (C_{\Aut(Q)}(Z(Q))/\Inn(Q)) \cdot \left\langle \theta \right\rangle$. 
Also from the second paragraph of the proof of Lemma 5.2 of \cite{PS18}, we have that 
$\Out_\CF(Q)=O^{p'}(\Out_\CF(Q)) \cdot \left\langle \theta \right\rangle$
where $O^{p'}(\Out_\CF(Q)) = \Out_\CF(Q)\cap (C_{\Aut(Q)}(Z(Q))/\Inn(Q))=C_{\Out_\CF(Q)}(Z(Q))$. Here we note that  by Lemma 4.12, Lemma 5.3 and  Theorem 5.9  of \cite{PS18},   the hypothesis  of Lemma   5.2 of \cite{PS18} holds.   As a result, we get that 
$$\Out_\CF(Q)=C_{\Out_\CF(Q)}(Z(Q))\cdot \left\langle \theta \right\rangle.$$
The group $C_{\Out_\CF(Q)}(Z(Q))$ acts trivially on $\Irr^3(Q)$, because if $\varphi\in C_{\Out_\CF(Q)}(Z(Q))$, then 
\begin{equation*}
\begin{split}
\mathbf {}^\varphi \tau_\phi(x)=\tau_\phi({}^\varphi x)&=\begin{cases}
x, & \text{if $x\in Z(Q)$ }\\
0, & \text{if $x\not\in Z(Q)$ }
 \end{cases}
		 \end{split} = \tau_\phi(x).
		 \end{equation*}	
On the other hand, $\left\langle \theta \right\rangle$ acts transitively on $\Irr^3(Q)$. Indeed, 
by (3B) of \cite{W}, we have that
\begin{equation*}
\begin{split}
\mathbf {}^\theta \tau_\phi(x)=\tau_\phi({}^\theta x)&=\begin{cases}
x^m, & \text{if $x\in Z(Q)$ }\\
0, & \text{if $x\not\in Z(Q)$ }
 \end{cases}
		 \end{split} = \tau_{\phi^m}(x).
		 \end{equation*}
where $m$ is a primitive $p^\text{th}$ root of unity. As a result, $\Out_\CF(Q)$ acts 
transitively on $\Irr^3(Q)$. The preceding argument gives us that for any $\chi \in \Irr^3(Q)$, 
$$I(\sigma_1, \chi)=\Stab_{\Out_\CF(Q)}(\chi)=O^{p'}(\Out_{\CF}(Q)) =\SL_2(p).$$		 
Since  $z(k\, \SL_2(p))=1$, the  contribution of  the trivial chain  to $\w_Q(\CF, 0, 3)$ is  $1$.
 Now $|\Irr^3(Q)|=p-1$  whereas  $|\Out_S(Q)|=1 $. Thus every  every element of $\Irr^3(Q)$  is $S$-stable   and it follows from Lemma~\ref{lem:max} that   $\sigma_2 $ has no contribution to   $\w_Q(\CF, 0, 3)$.
\end{proof}

\begin{Proposition}\label{weightqS7}
Suppose that  $p=7$, $Q$ is $\CF$-centric-radical  and  $\Out_{\CF}(Q) \cong C_3 \times 2\cdot S_7$. Then   $ \mathbf w_Q(\CF, 0, 3)=6$.
\end{Proposition}

\begin{proof}  Arguing as in the proof of Proposition~\ref{weightq}, we  have that for any $\chi \in \Irr^3(Q)$, 
$$I(\sigma_1, \chi)=\Stab_{\Out_\CF(Q)}(\chi)=O^{p'}(\Out_{\CF}(Q)) =2\cdot A_7.$$
Since  $z(k\, 2\cdot A_7)=6$  (see for instance the ATLAS of finite groups \cite{CCNPW}),  the  contribution of  the trivial chain  to $\w_Q(\CF, 0, 3)$ is  $6$.  The non-trivial chain has no contribution by the same argument as in Proposition~\ref{weightq}.
\end{proof} 

\iffalse

\begin{Proposition}\label{trivial Q}   Suppose that   $Q$  is $\CF$-centric, radical.  If  $\Out_{\CF}(Q) =\GL_2 (7) $, then the contribution of the  trivial  chain to  $\w_Q(\CF, 0, 5)$  equals $20$.
\end{Proposition}

\begin{proof}  By  the  proof of Lemma  5.2 of \cite{PS18},  $Q/Z(Q) \cong { \mathbb F}^4 $   is a  simple, faithful  ${\mathbb F}_7\Out_{\CF} (Q)$-module, hence so is  the dual module   $V=\Irr(Q/Z(Q)) \cong  (Q/Z(Q))^*$.  The result follows from Proposition~\ref{prop:GL2-orbits}  since  $z(kC_3)=3 $, $z(kC_6)=6 $, $z(kC_2)= 2$, $z(kS_3)=3 $ and $ z(k\GL_2(7)) = 6 $.
\end{proof}
\fi

\begin{Proposition}\label{trivial Q_general}   Suppose that  $p\geq 7 $  and    $Q$  is $\CF$-centric-radical.  If  $\Out_{\CF}(Q) \cong \GL_2 (p) $, then the contribution of the  trivial  chain to  $\w_Q(\CF, 0, 5)$  equals $2p+6$.
\end{Proposition}

\begin{proof}
By Lemma~\ref{lem:outQ}, $Q/[Q,Q]$ is  an absolutely  simple  and faithful  ${\mathbb F}\Out_{\CF} (Q)$-module, hence so is  the dual module   $V:= \Irr^5(Q) =\Irr(Q/[Q,Q]) \cong  (Q/[Q,Q])^*$.  The result follows from Proposition~\ref{prop:GL2-orbits_general}  since  $z(kC_3)=3 $, $z(kC_{p-1})=p-1 $, $z(kC_2)= 2$, $z(kS_3)=3 $ and $ z(k\GL_2(p)) = p-1 $.
\end{proof}

\begin{Proposition} \label{doublecover}  Suppose that  $Q$ is $\CF$-centric-radical. If  $\Out_{\CF} (Q) \cong   C_3 \times 2.S_7$, then  the contribution of  the trivial chain to   $\w_Q(\CF, 0, 5) $   equals  $36 $.
\end{Proposition}

\begin{proof}   Let $G=   \Out_{\CF}(Q)  = Z \times  G_1 $, with $Z \cong C_3 $   and $G_1 \cong  2.S_7$.  Let   $ G_0=[G,G]  \cong 2.A_7 $.  As above,  $V:=\Irr^5(Q)  \cong  (Q/[Q,Q])^* $   is  a faithful and    absolutely  simple  ${\mathbb F}  G $-module.
We first consider $G_0$-orbits on $V$. Since  $|V|=7^4 $  and   $|G_0|= 7! $, there  is no free $G_0$-orbit.   The only involution of $G_0$ is  the  central element of  order $2$ which   acts as  $-1 $, hence no subgroup  of $G_0$ of even order fixes  a point  of $V\setminus \{0\} $. 
By the  ATLAS of Brauer characters,   the restriction of  $V$ to  a cyclic subgroup  of order  $5$ of $G_0$  does not contain the trivial submodule, hence  no element of order $5$    fixes  a  point   of $V\setminus\{0\} $.    Similarly,    if  $x \in  2.A_7 $  is $3$-element whose image in $A_7 $   is a $3$-cycle, then $x$ does not fix any  point  of $V\setminus\{0\} $  whereas if  the image of $x$  is a product of two disjoint cycles of  length $3$, then  $x$ does  fix a point of $V\setminus \{0\} $.  It follows that if $H$  is  a non-trivial subgroup  of $G_0$  fixing a  point  of  $V\setminus\{0\} $, then   $H$  is  of order  $7$, $21$ or   $3$.  In particular,  $H$ has a  normal Sylow $7$-group  and  consequently no two distinct Sylow subgroups  have a common  fixed point. 
By Lemma  3.2 (f) of \cite{PS18},    a generator  of   $S/Q $  acts    via  a  single Jordan block on  $Q/Z$ and  consequently $V$   has a  unique  simple ${\mathbb F}_7 S/Q $-submodule.  It follows that every Sylow  $7$-subgroup of $ G_0$  fixes  $6$  points of $ V \setminus\{0\} $.    Since   $G_0$    has   $5!$ Sylow  $7$-subgroups,  the  subset, say $W$  of $V\setminus \{0\} $  consisting of elements fixed by some non-trivial   $7$-element has size $6 \times 5!= 6!  <   7^4-1$.     Thus there is  at least one orbit, say $W'$  on $V\setminus\{0\} $    with stabiliser of order $3$     and  with  $|W'| = \frac{|G_0|}{ 3 } $.    It follows that    $W'$ is the unique $G_0$    orbit   on $ V\setminus W $.    Consequently, $W' $ and $W$   are both   $G$-invariant  and     since $W'$ is a single  $G_0$ -orbit,  it is a single $G$-orbit.    If $ w \in W  $, then $\Stab_{G_0} (w) $   and  hence $\Stab_G(w) $ has a non-trivial normal  $7$-subgroup    hence  $z(k\Stab_G(w)) = 0 $.    Now suppose that $w \in W' $.  Then  by the Frattini argument,   $G=  G_0 \Stab_{G}    (w) $.  Since $\Stab_{G_0}(w)  \cong C_3 $,   $ |\Stab_G(w) | = 18 $.    The Sylow $3$-subgroups of  $G$ are of the form $Z \times P$, where $P \cong C_3 \times C_3$ is a Sylow    $3$ subgroup   of 
$G_0$.     Since    $C_{G_0} (P) =  P $ and by calculation verified with \textsc{gap}, we have  that  $\Stab_G(w)  = X \times   (Y \rtimes \left\langle \sigma \right\rangle) $ where  $X$ and  $Y$  are of order  $3$  and   $\sigma $ acts by inversion on $Y$.    Thus $z (k\Stab_G(w)) =  9 $.      Since $2.S_7 $  has  $9$ irreducible characters of   with zero $7$ defect,  $ z(k G)  =  27 $. This proves the result.
 \end{proof}

\begin{Proposition}\label{non-trivial Q}   Suppose that $p=7 $.   Let  $\Irr^7(Q)'  $  be the subset  of  $\Irr^7(Q)$    consisting of those characters which are not  $S$-stable,  let   $J=N_{\Out_{\CF} (Q)} (\Out_S(Q))  $ and let $w= \sum_{ v \in \Irr^7(Q)'  /J}   z(k\Stab_{J} (v))$.
\begin{enumerate} \item [(i)]  If $\Out_{\CF}(S)   \cong C_6 $, then  $w=110 $.
\item[(ii)]  If $\Out_{\CF}(S)   \cong C_6  \times C_2$, then  $w=56$.
\item[(iii)]  If $\Out_{\CF}(S)   \cong C_6  \times C_3$, then  $w=67$.
\item[(iv)]  If $\Out_{\CF}(S)   \cong C_6  \times C_6$, then  $w=41$.
\end{enumerate} 
 \end{Proposition}
 
\begin{proof}    The  map    which sends  an element $A$  of $\GL_2({\mathbb F}) $ to the element   $( \det(A),  A) $   of $L$ induces an isomorphism $\varphi: \GL_2( {\mathbb F}  )  \to  L /K $.  Further,   it follows from the  explicit description of  the  ${\mathbb F}L$-module structure on $V:=Q/Z(Q)$ given in \cite{PS18}  that  through this  isomorphism  $ V $  is identified with   the   faithful  ${\mathbb F} \GL_2({\mathbb  F}) $-module  of Section $3$   corresponding to  $r=1 $ and consequently   $ \Irr^7(Q) = \Irr (Q/Z(Q) ) \cong V^*$   is identified with the  faithful  ${\mathbb F} \GL_2({\mathbb  F}) $-module  of Section $3$   corresponding to  $r=5$.      
  
By  Proposition~\ref{prop:descF_general}, $J=  \Out_{B'} (Q)  $.  Also,   $ Q $ is normal in $B$ and $QC_{B}(Q)  = QK =QC_{B'}(Q)$ since $K\leq B'$. Hence,  $\Out_{B'}(Q) =  B'/QC_{B'}(Q)    \cong  (B' \cap B_0) /K  \leq L/K  $   where the  isomorphism follows from the  fact that the kernel of  the natural map  $  B'\cap B_0 $ to  $(B' \cap B_0)Q /K Q $    has kernel $(B'\cap B_0) \cap  KQ = K$.  Thus, we may replace $J$ in the definition  of $w$ by   $K (B' \cap B_0) =   KS_0(B'\cap D)$. Let $I, I_i, 4 \leq i \leq  6 $   be as in   Proposition~\ref{prop:UTi-orbits}.
 It is straightforward to check that  $ \varphi^{-1}  (B'\cap B_0)/K) =I_4 $ in case (i),   $ \varphi^{-1}  (B'\cap B_0)/K) =I_5$ in case (ii),
  $ \varphi^{-1}  (B'\cap B_0)/K) =I_6$ in case (iii)  and  $ \varphi^{-1}  (B'\cap B_0)/K) =I$ in case (iv). 
 Proposition~\ref{prop:UTi-orbits} now  yields the  claimed   value   of  $w$.
\end{proof} 

\subsection{Results for $S$}

\begin{Proposition}   \label{prop:S non-linear_general}   $cd(S) =  \{1, p, p^2\}  $.
\begin{enumerate} 
\item [(i)] $S$ has $p^2$ linear characters.
\item [(ii)] $S$ has $p(p-1)$ characters of degree $p^2$ and these characters cover $S$-stable irreducible
characters of $Q$ of the same degree.
\item [(iii)]  $S$ has  $(p^3-1)$ characters of degree $p$    and these characters cover    non $S$-stable  linear characters of  $Q$.
\end{enumerate}
\end{Proposition}

\begin{proof}
Part (i) is clear since $[S, S]=Z_4$. Let  $\chi \in \Irr(S)$ and $\lambda \in \Irr(Q)$ be  covered by $\chi$.
If $\lambda$ is $S$-stable, since $S/Q$ is cyclic of order $p$,  then $\chi $ is an extension of  $\lambda$  and  there are  $p$ such extensions.  If $\lambda$ is  not $S$-stable then  $\chi=\Ind_Q^S (\lambda)$   and  $\chi$ covers $p$  irreducible characters of $Q$.
Recall  from  Proposition \ref{weightw_general}  that $cd(Q) =  \{1, p^2\} $ and that   $Q$ has $(p-1)$ irreducible characters of degree $p^2$.  Since $|S/Q| =p $,  every irreducible character of $Q$  of degree $p^2 $ is $S$-stable  and we obtain   $p(p-1) $ characters  of  $S$ of degree $p^2 $ as in part (ii).     By part (i),  $S$ has $p^2 $ linear characters, hence  $Q$  has  $p$-linear characters that are  $S$-stable. By Proposition \ref{weightw_general} , $Q$ has $p^4$ linear characters.  So  there are $p^4-p $ linear characters of $Q$ which are non $S$-stable  and  these  give rise to    $(p^3-1)$ characters of   $S$ of degree $p$.  
\end{proof}

\begin{Proposition}\label{weight: S linear_general}
We have that  $\w_S(\CF, 0, 6)=p^2$ if $\Out_\CF(S)\cong C_{p-1}\times C_{p-1}$.
\end{Proposition}

\begin{proof} Recall  the set  up of   Proposition~\ref{prop:descF_general} and  note that $[S, S]=\Phi(S)=Z_4$ and  that  $D$ acts on $S/[S,S]\cong \F^2$ by scalars. Thus, $S/[S,S]  \cong  \Irr(S/[S,S])$ as   ${\mathbb F}_p  \Out_{\CF}(S)$-modules  and  it suffices to determine the $\Out_\CF(S) =\Out_{B'\cap D}$-orbits of $V:=S/[S,S]$. 
By  Proposition~\ref{prop:descF_general},  $B'\cap D=D$.  So,   $(0, 0)$ is the unique element in $V$ 
which is stabilized under the action of $B'\cap D$. If $a$ or $b$ is zero, the stabilizer of $(a, b)$ under the action of $D$ is 
\[   \bigg\{  \biggl(t,   \begin{pmatrix} 1 & 0\\0& 1 \end{pmatrix} \biggr),  t  \in {\mathbb F}^{\times} \bigg\} \text{ or }  \bigg\{  \biggl(1,   \begin{pmatrix} \alpha & 0\\0& 1 \end{pmatrix} \biggr), \alpha  \in {\mathbb F}^{\times} \bigg\}\]
and there are $2$ distinct orbits with this property. Moreover, the $D$-orbit of $(1, 1)$ includes all elements $(a, b)$ where both $a$ and 
$b$ are non-zero. Hence, $\w_S(\CF, 0, 6)=(-1)^0(1\cdot z(k (C_{p-1}\times C_{p-1}))+2 \cdot z(k C_{p-1})+1 \cdot z(k))=p^2$.
\end{proof}

\begin{Proposition}\label{weight: S linear} Suppose that $p=7 $. We have that $\w_S(\CF, 0, 6)=14$ if $\Out_\CF(S)\cong C_6$,  $\w_S(\CF, 0, 6)=19$ if $\Out_\CF(S)\cong C_6\times C_2$, and 
$\w_S(\CF, 0, 6)=26$ if $\Out_\CF(S)\cong C_6\times C_3$.
\end{Proposition}

\begin{proof}
As above, it suffices to determine the   $\Out_\CF(S) =\Out_{B'\cap D}$-orbits of $V:=S/[S,S]$.     If $\Out_{\CF}(S)\cong C_6$, then by Proposition~\ref{prop:descF_general}   $B'\cap D= D' $. 
An  element $(a, b)\in V$ has trivial stabilizer under the action of $D'$ if and only if $a$ or $b$ is non-zero. Hence there 
are $8$ $D'$-orbits of $V$ which have trivial stabilizer. The remaining element which is $(0,0)$ has stabilizer equal 
to $D'$. Hence, in this case $\w_S(\CF, 0, 6)=(-1)^0(1\cdot z(k C_6)+8\cdot z(k))=14$.

If $\Out_{\CF}(S)\cong C_6 \times C_2$, then  by Proposition~\ref{prop:descF_general}  
\[   B'\cap D=  \bigg\{  \biggl(\pm \alpha,   \begin{pmatrix} \alpha & 0\\0& 1 \end{pmatrix} \biggr),  \alpha  \in {\mathbb F}^{\times} \bigg\}.\] 
 Note that $(0, 0)$ is the unique element in $V$ which 
is stabilized under the action of $B'\cap D$. If $a$ or $b$ is zero, the stabilizer of $(a, b)$ under the action of $B'\cap D$ is 
equal to  one of 
\[  \bigg\{  \biggl(\pm 1,   \begin{pmatrix} 1 & 0\\0& 1 \end{pmatrix} \biggr)   \bigg\} \text{ or }   \bigg\{  \biggl( 1,   \begin{pmatrix} \pm 1 & 0\\0& 1 \end{pmatrix} \biggr)   \bigg\}  \] 
and there are $2$ distinct orbits with this property. On the other hand, if both $a$ and $b$ are non-zero, then $(a, b)$ has trivial 
centralizer. There are $3$ such orbits. Hence, we have that $\w_S(\CF, 0, 6)=(-1)^0(1\cdot z(k (C_6\times C_2))+2 \cdot z(k C_2)+3 \cdot z(k))=19$.

If $\Out_{\CF}(S)\cong C_6 \times C_3$, then 
\[   B'\cap D=  \bigg\{  \biggl(2^k \alpha,   \begin{pmatrix} \alpha & 0\\0& 1 \end{pmatrix} \biggr), k\in \{0, 1, 2\}, \alpha  \in {\mathbb F}^{\times} \bigg\}.\] 
 Similarly, $(0, 0)$ is the unique element in $V$ 
which is stabilized under the action of $B'\cap D$. If $a$ or $b$ is zero, the stabilizer of $(a, b)$ under the action of $B'\cap D$ is 
equal to 
\[  \bigg\{  \biggl(\omega^k,   \begin{pmatrix} 1 & 0\\0& 1 \end{pmatrix} \biggr), k \in \{0, 1, 2\}  \bigg\} \text{ or }  \bigg\{  \biggl(1,   \begin{pmatrix} \omega^k & 0\\0& 1 \end{pmatrix} \biggr), k \in \{0, 1, 2\} \bigg\} \] 
and there are $2$ distinct orbits with this property. On the other hand, if both $a$ and $b$ are non-zero, then $(a, b)$ has trivial 
centralizer. There are $2$ such orbits. Hence, we have that $\w_S(\CF, 0, 6)=(-1)^0(1\cdot z(k (C_6\times C_3))+2 \cdot z(k C_3)+2 \cdot z(k))=26$.
\end{proof}

\begin{Proposition} \label{prop:S5}  Suppose that  $p=7$. We have that $\w_S(\CF, 0, 5)=110$ if $\Out_\CF(S)\cong C_6$,  $\w_S(\CF, 0, 5)= 56 $ if $\Out_\CF(S)\cong C_6\times C_2$, $\w_S(\CF, 0, 5)= 67$ if $\Out_\CF(S)\cong C_6\times C_3$  and   $\w_S(\CF, 0, 5)=41$ if $\Out_\CF(S)\cong C_6\times C_6$. \end{Proposition}

\begin{proof}
This follows from Lemma \ref{weight:Q-S},  Proposition \ref{weightw_general} and Proposition \ref{non-trivial Q}. \end{proof}

\begin{Proposition} \label{prop:S4}
Suppose that $p=7$. We have that $\w_S(\CF, 0, 4)=7$ if $\Out_\CF(S)\cong C_6$, $\w_S(\CF, 0, 4)=14$ if $\Out_\CF(S)\cong C_6\times C_2$, 
$\w_S(\CF, 0, 4)=5$ if $\Out_\CF(S)\cong C_6\times C_3$ and $\w_S(\CF, 0, 4)=10$ if $\Out_\CF(S)\cong C_6\times C_6$.
\end{Proposition}

\begin{proof}   Note again that  $\Out_\CF(S)=\Out_{B'\cap D}(S) $.
Recall from Proposition \ref{prop:S non-linear_general} that  $$\Irr^4(S)=\{ \chi\in \Irr(S)\ | \ \chi \text{ covers some } \lambda \in \Irr^3(Q)\}.$$ 
Let $\Irr^3(Q)=\{\lambda_i: i=1, \ldots, 6 \}$. Then $\Irr^4(S)=\bigcup_{i=1}^6 A_{\lambda_i} $ where 
$A_{\lambda_i}$ is the set of characters of $S$ covering $\lambda_i$. Clearly $|A_{\lambda_i}|=7$.  
Since each $\lambda_i$ is faithful on $Z(Q)=\left\langle x_6 \right\rangle$, and $\lambda_i(x)=0$ for all $x\in Q \backslash Z(Q)$, 
each $A_{\lambda_i}$ is $B'\cap D_1$-stable where 
\[D_1=\Stab_{D}(x_6)=\bigg\{  \biggl(t,   \begin{pmatrix} \alpha & 0\\0& 1 \end{pmatrix} \biggr),  t, \alpha  \in {\mathbb F}^{\times}, t^2\alpha^3=1 \bigg\}. \] 
Note that $B'\cap D_1$ is non-trivial unless $\Out_\CF(S)\cong C_6$. 

Let $\Out_\CF(S)\cong C_6\times C_6$. Then $B'\cap D=D$ and $B'\cap D_1=D_1$. Let us take an element $\chi\in A_{\lambda_i}$, since $S/Q=\left\langle x_1 Q\right\rangle$, 
and $(-1, I)$ acts trivially on $x_1$ we have $\left\langle (-1, I)\right\rangle \leq \Stab_{D_1}(\chi)\leq D_1$. So $\Stab_{D_1}(\chi)$ has order either 
$2$ or $6$. But if all $\chi$'s have stabilizer of order $2$ then it follows that $3$ divides $7$, which is impossible. So there is at least one $\chi$ such 
that $\Stab_{D_1}(\chi)=D_1$. In fact, $\Stab_{D}(\chi)=D_1$ since the elements $\biggl(t,   \begin{pmatrix} \alpha & 0\\0& 1 \end{pmatrix} \biggr)$ where $t^2\neq 1$ or $\alpha^3\neq 1$   do not stabilize $\lambda_i$.  The other elements of $A_{\lambda_i}$ can 
be expressed as $\chi \otimes \eta$ where $\eta$ is a non-trivial irreducible character of $S/Q$, so an element of $D_1$ stabilizes $\chi \otimes \eta$ 
if and only if it stabilizes $\eta$ or equivalently centralizes $x_1$. Hence, $\Stab_D(\chi \otimes \eta)=\left\langle (-1, I)\right\rangle \cong C_2$. Since there are 
$36$ such elements, there are $2$ orbits with $C_2$ as a stabilizer and it follows that $\w_S(\CF, 0, 4)=1\cdot 6 + 2 \cdot 2=10$ in this case.
% \blue{\textbf{[Without checking all details, it looks like there will be a single orbit with stabiliser $C_{p-1}$ and $2$ orbits with 
%stabiliser  $C_2$ on $\Irr^4(S)$, giving $w_S(\CF,0,4)=1 \cdot (p-1)+2 \cdot 2=p+3$ in general.]}}

Let $\Out_\CF(S)\cong C_6\times C_2$. Then $B'\cap D=C_6\times C_2$ and $B'\cap D_1=\left\langle (-1, I)\right\rangle\cong C_2$. For all $\chi\in \Irr^4(S)$, 
by the arguments from the previous paragraph we get that $\Stab_{B'\cap D}(\chi)=B'\cap D_1$. Hence, in each orbit there are $6$ elements, and so there 
are $7$ different orbits. It follows that $\w_S(\CF, 0, 4)=7\cdot 2 =14$ in this case. 

Let $\Out_\CF(S)\cong C_6\times C_3$. Then $B'\cap D=C_6\times C_3$ and $B'\cap D_1= \biggl\langle \biggl(1,   \begin{pmatrix} 2 & 0\\0& 1 \end{pmatrix} \biggr) \biggl\rangle\cong C_3.$
Since $A_{\lambda_i}$ is $B'\cap D_1$-stable, for each $\chi \in A_{\lambda_i}$, we have that $\Stab_{B'\cap D_1}(\chi)\leq C_3$. If all such $\chi$'s have trivial stabilizer then  it 
follows that $3$ divides $7$, a contradiction. So there is at least one $\chi\in A_{\lambda_i}$ with $\Stab_{B'\cap D_1}(\chi)=B'\cap D_1.$ In fact, 
$\Stab_{B'\cap D}(\chi)=B'\cap D_1$ since the elements $\biggl(t,   \begin{pmatrix} \alpha & 0\\0& 1 \end{pmatrix} \biggr)$ where $t^2\neq 1$ or $\alpha^3\neq 1$   do not stabilize $\lambda_i$. The other elements of $A_{\lambda_i}$ can 
be expressed as $\chi \otimes \eta$ where $\eta$ is a non-trivial irreducible character of $S/Q$, so an element of $B'\cap D_1$ stabilizes $\chi \otimes \eta$ 
if and only if it stabilizes $\eta$ or equivalently centralizes $x_1$, so $\Stab_{B'\cap D}(\chi\otimes\eta)=1$. Since there are $36$ such elements and each orbit 
has $18$ elements in it, there are two orbits with trivial stabilizer. Hence, we have that $\w_S(\CF, 0, 4)=1\cdot 3+ 2 \cdot 1=5$.

Let $\Out_\CF(S)\cong C_6$. Then $B'\cap D=C_6$ and $B'\cap D_1=1$. Then it follows that each $\chi\in \Irr^4(S)$ has trivial stabilizer, or equivalently 
each orbit has $6$ elements. Hence, there are $7$ orbits with trivial stabilizer and hence  $\w_S(\CF, 0, 4)=7\cdot 1=7$.
\end{proof}

\subsection{Results for $R$}

 Recall from \cite{PS18} that  $[R, R] = \Phi(R)=   Z_3 $,  $Z_3 $ is  elementary abelian and $R/Z_2  $ is extra-special of order $p^3 $ and exponent $p$.
 
\begin{Proposition}   \label{prop:R non-linear}
We have $cd(R) = \{1, p\} $. Let   $\lambda \in \Irr(Z_3)  $  and  let $\mu  $ be the restriction of $\lambda $ to $Z_2$.
\begin{enumerate} \item [(i)] If $\lambda =1 $, then $\Irr(R|\lambda) =\Irr(R/Z_3) $  is the set of linear characters  of $R$.
\item [(ii)] If $\mu =1 $  and $\lambda \ne 1 $, then    $\Irr(R | \lambda) =\{\chi\} $,  where  $\chi =\Ind_{Z_4}^ R (\theta) $    and   $\theta $ a  linear character of  $Z_4 $ covering $\lambda$. Further, $\lambda  $ and hence $\chi$   is $S$-stable.
\item  [(iii)] If  $\mu \ne 1 $, then  $\Irr(R| \mu ) =\Irr(R|\lambda) $ consists of  $p$ characters  each of   degree $p$.   Setting $Z_{\mu}= \mathrm{Ker}(\mu) $  and  $ Y = C_{R} (Z_2/ Z_{\mu})  $,   $Y$   is of index $p$  in $R$  and   $Y/Z_{\mu} $   is abelian.
Every element of $\Irr(R |  \mu)$ is  of the form $\Ind_{Y}^R (\theta)$ where   $\theta $ is a linear character of  $Y$ covering $\mu$.
%In particular, every element of $\Irr(R| \mu )$ is $\Out_S(R)$-stable  and  $\Irr(R|\lambda)  =\Irr(R|\mu)$.
%Every element of $\Irr(R |  \mu)$ is  of the form $\Ind_{Z_4}^R (\theta)$ where $\theta $ is a linear character of  $Z_4$ covering $\mu$. In particular, every element of $\Irr(R| \mu )$ is $\Out_S(R)$-stable  and  $\Irr(R|\lambda)  =\Irr(R|\mu)$.
\item [(iv)]  Suppose that  $\mu \ne 1 $   but  the restriction of $\mu $ to $Z_1  $ is trivial.  Then  every element of $\Irr(R| \mu )$ is $\Out_S(R)$-stable.
%\item [(iv)]  Suppose that   $\mu \ne 1 $   but  the restriction of $\mu $ to $\left\langle x_2\right\rangle $ is trivial.  Then  every element of 
%$\Irr(R| \mu )$ is  stabilised by    the element  $(-1,  \begin{pmatrix} 1 & 0 \\ 0& 1\end{pmatrix})$, $t, \alpha \in  {\mathbb F}$  
%with  $t^2\alpha^3 =1 $.
\end{enumerate}
\end{Proposition}

\begin{proof}    Part (i)  is immediate from the fact that $Z_3=[R,R]$. The first  assertion of  part (ii)    follows  from  the fact that  $R/Z_2$ is extra-special of order $p^3$  with $Z_3/Z_2 = Z(R/Z_2) $  and $Z_4/Z_2  $ is  a maximal subgroup of $R/Z_2$.      Since $Z_3/Z_2 \leq  Z(S/Z_2)$, $\lambda $ is  $S$-stable   and since $\chi $ is the unique irreducible character of $R$ covering $\lambda $  it follows that $\chi $ is  also $S$-stable. 

Next we prove (iii).  Recall   from Sections 3 and 4     of \cite{PS18}   the generators $x_1, x_3,  x_4, x_5, x_6 $ of $R$.  Suppose that  $\mu \ne 1 $.  By \cite[Lemma~4.5]{PS18}, $\Aut (R) $    contains  a subgroup  isomorphic to $\SL_2(p)$ acting  faithfully on   $Z_2  = \mathbb{F}^2$. In particular,  $\Aut (R) $   acts transitively  on  the  non-identity  elements  of $ Z_2 $. Hence  it suffices  to consider the case  that  the restriction of  $\mu  $  to  $Z_1 $ is trivial.
 The  group   $Z_4/Z_1$ is abelian, hence  all elements  of   $\Irr(Z_4 |\mu ) $  are    linear. Let $\theta \in  \Irr(Z_4 |\mu ) $.  Since $ x_5  \in  [\left\langle x_1 \right\rangle,  \left\langle x_4 \right\rangle]    $  and $\mu  $ is  faithful on $Z_2/Z_1 $,   $x_5 $ does not stabilise $\theta$  and  by the maximality of $Z_4 $ in $R$  we have that   $Z_4 = \Stab_R (\theta ) $. Hence   $\Ind_{Z_4}^R (\theta) $ is   the unique  irreducible  character of $R$  covering   $\theta $  and the $R$-conjugacy class of  $\theta $  is of size $p$.   Since $|Z_4/Z_2|=p^3 $,   $|\Irr(Z_4 |\mu ) | =p^2 $,  and  hence there  are  $p$ irreducible characters of $R$ covering elements of $\Irr(Z_4|\mu)$. Since  $Z_2$ is central in $R$,  $\mu$ is $R$-stable and  consequently  $\Irr(R|\mu)$   is  equal to the set of irreducible characters of  $R$ covering elements of  $\Irr(Z_4|\mu)$.  It remains to  show that $\Irr(R| \mu ) =\Irr(R|\lambda) $. Clearly, $ \Irr(R|\lambda) \subseteq \Irr(R|\mu) $.  Let  $\chi \in \Irr(R|\mu)$. Then $\chi $ covers  one of the $p$ extensions of  $\mu $ to $Z_3 $, say $\chi $ covers  $\lambda'$ .  Again since $ x_5  \in  [\left\langle x_1 \right\rangle,  \left\langle x_4 \right\rangle]    $  and $\mu  $ is  faithful on $Z_2/Z_1 $,    $\lambda' $  is not $R$-stable   and  therefore has at least $p$  conjugates in $R$  each of which is an extension of $\mu$.  Thus  $\lambda $ is $R$-conjugate to $\lambda'$  which means that $\chi $ also covers $\lambda $. So, $\Irr(R|\mu) \subseteq \Irr (R|\lambda)$. This proves (iii).

Suppose  $\mu\ne 1 $ has trivial restriction  to $Z_1 $  and let  $ \chi \in \Irr (R|\mu) $. By  the proof of part (iii),     $\chi= \Ind_{Z_4}^R (\theta)  $ where $\theta  $ is   above  a linear character of $Z_4 $ covering $\mu$.  Since $Q/Z_1 $ is abelian,  $\,^{x_2}\theta =\theta $  and hence
$$\,^{x_2} (\chi) = \Ind_{Z_4}^R (\,^{x_2} \theta)  = \Ind_{Z_4}^R (\theta)  =\chi. $$
Since the image of $x_2$ generates $\Out_S(R)$  we obtain (iv).  The  statement  on character degrees  follows from (i)-(iv).

%Let   $\mu $  be as in (iv)  and let $Y= \left\langle  Z_3,  x_1\right\rangle $. Then $Y/\left\langle x_5\right\rangle$  is a  maximal abelian subgroup
% of $R/\left\langle x_1\right\rangle$.   Let $\theta $ be an extension of  $\mu $  to   $Y$. Then $d$ normalises  $Y$ and  
%stabilises $\theta $,  hence $d$ normalises   $\chi =\Ind_Y^R\theta  $.    
\end{proof}

\begin{Lemma} \label {lem:outR}  Suppose that  $p\geq 5$ and $ R$ is $\CF$-centric radical.  
\begin{enumerate} [(a)]
\item  The natural map   $\Out_{\CF}(R) \to \Aut_{\CF} (R/\Phi(R))  \cong \GL_2(p)  $ is injective  and identifying $\Out_{\CF}(R) $ with its image in   $\Aut_{\CF} (R/\Phi(R)) $,
$$ \SL_2(p)  =O^{p'} (\Out_{\CF}(R))  \leq  \Out_{\CF}(R) \leq   \GL_2(p).$$  
 \item  If   $\Out_{\CF}(R)  \cong   \GL_2(p) $  then $ \Out_{\CF}(S) \cong C_{p-1} \times C_{p-1} $.
\item  If $p=7$ and   $\Out_{\CF}(R) \cong \SL_2(7).2 $, then   $\Out_{\CF}(S)   \cong C_6 \times C_2$.
 \end{enumerate} 
 Moreover, if $ p > 7 $,  then $\Out_{\CF}(R) \cong   \GL_2(p) $  and if  $p=7 $, then either  $\Out_{\CF}(R)  \cong  \GL_2(p) $ or  $\Out_{\CF}(R) \cong \SL_2(7).2 $
\end{Lemma} 
 \begin{proof}   By Proposition~\ref{martin},  $R$ is $\CF$-essential.  Part~(a)   is  given by  \cite[Lemma~4.5] {PS18}.
 If  $ p\ne 7 $  or  if   $p=7 $ and no element of $\CW$ is $\CF$-essential, then the  rest of the result follows  by   \cite[ Lemma 4.12, Lemma 5.4, Lemma 5.5]{PS18}.  If $ p=7 $ and some element of $\CW$ is $\CF$-essential, then the  rest of the result follows by \cite[Lemma~5.12, Theorem~5.16]{PS18}.
 \end{proof}

\begin{Proposition} \label{weight: R} Suppose that $p \ge 5$ and $R$ is $\CF$-centric-radical. We have that $\w_R(\CF, 0, 5)=0$.
\end{Proposition}

\begin{proof}
%We have $\SL_2(p) \leq \Out_\CF(R) \leq \GL_2(p) $ and by  the proof of Lemma 4.5 of \cite{PS18}, $V:=R/[R, R]\cong \mathbb{F}^2$ is %a faithful $\mathbb{F}_p \Out_\CF(R)$-module.   Since  $O_p(\SL_2(p)) =1 $  this means in particular that 
% $\Res_{\SL_2(p)}^{\Out_\CF(R)} V$ is simple.  
We use  Lemma~\ref{lem:outR}. There are two chains in $\CN_R$; namely $\sigma_1=(1)$ or $\sigma_2=(1<U)$ 
where $U$ is the subgroup of lower  triangular unitriangular matrices. The result follows  by the same argument as  in Proposition~\ref{prop:W}  since $\Irr^5(R) = \Irr(R/\Phi(R) ) \cong  (R/\Phi(R))^* \cong {\mathbb F}^2 $   is a faithful  simple ${\mathbb F}_p \Out_{\CF}(R)$-module.
 \end{proof}

 Recall that by Proposition~\ref{prop:descF_general},  $N_{\Aut_{\CF}(R) }(\Aut_S(R))  = \Aut_{B'} (R) $  and $ N_{\Out_{\CF}(R) }(\Out_S(R))  = \Out_{B'} (R) $.

\begin{Proposition}   \label{wR(F,0,4)_general}  Suppose that $R$ is  $\CF$-centric-radical. 
  If $\Out_{\CF} (R) \cong   \GL_2(p)$, then  the contribution of  the trivial chain to   $\w_R(\CF, 0, 4) $   equals  $1$.  If $p=7 $ and  
$\Out_{\CF} (R) \cong   \SL_2(7).2$, then  the contribution of  the trivial chain to   $\w_R(\CF, 0, 4) $   equals  $3$.
\end{Proposition} 

\begin{proof}  By Lemma~\ref{lem:outR}  and    Proposition~\ref{prop:descF_general}, if   $\Out_{\CF}(R)  =  \GL_2(p) $  then    $D'=D$ and  if   $p=7 $,  and    $\Out_{\CF}(R) =\SL_2(7).2 $, then    $D'$ is the subgroup of  $D$ defined in   Proposition~\ref{prop:descF_general} (iv).

Identifying $\Aut (R/[R,R])$    with  $\GL_2(p) $   via the  basis   $ x_3 +[R, R], x_1 +[R,R]$  and noting that   $\Out_S(R)$ is  generated by   the  image of conjugation by $x_2$,   we have that  $\Out_S(R)$   is the subgroup   $U$  of  $\GL_2(p) $ of lower  unitriangular matrices   and the  image  of  an   element  $\biggl( t,    \begin{pmatrix} \alpha & 0\\0& 1 \end{pmatrix} \biggr) $  of $D'$    in $\Out_{\CF}(R) $     is the  diagonal matrix    $\begin{pmatrix} t \alpha  & 0\\0& \alpha  \end{pmatrix}  $.   

Let  $A_1 $ be  the  subset of $\Irr^4(R)$  consisting of characters covering  the trivial  character of $Z_2$ and  set  $A_2 =  A\setminus A_1$. Then  clearly  $ \Irr^4(R)= A_1  \cup A_2 $ is an $\Out(R) $-stable  partition of $A$ and hence an $\Out_{\CF}(R)$-stable partition of  $A$.

By   Proposition~\ref{prop:R non-linear}(ii)  the elements of $A_1 $ are in  $\Out(R)$ equivariant  bijection with     the set  of  irreducible non-trivial  characters of $Z_3/Z_2 \cong {\mathbb F}$.    If $\Out_{\CF}(R)  =\GL_2(p)$ then  $D'=D$ acts transitively   on  $\Irr(Z_3/Z_2) $ and hence on $A_1$.
Since $|A_1|=p-1 = |\GL_2(p):\SL_2(p)| $, it follows that  $\Stab_{\Out_{\CF}  (R)} (\chi  )  =\SL_2(p)  $ for $\chi \in  A_1 $. Thus $A_1$ contributes $1$  to the  $\sigma_1 $-component of  $\w_R(\CF, 0, 4)$. Now suppose that   $\Out_{\CF}(R)  =\SL_2(p):2$. 
Then  every  $\Out_{B'\cap D } (R)$-orbit of $A_1$  is of size at least $2$.   On the other hand,   since  
$|\Out_{\CF}(R):\SL_2(p)| =2$,  the maximum  size of  an $\Out_{\CF}(R)$-orbit of  $A_1 $  is $2$. It follows that there are three $\Out_{\CF}(R)$  orbits  on $A_1$    each of length $2$  and   each of them  has  $\SL_2(7)$ as stabiliser. Thus $A_1$ contributes $3$  to the  $\sigma_1 $-component of  $\w_R(\CF, 0, 4)$. 

Next we consider  the  contribution of $A_2 $ to the  $\sigma_1$-component of $\w_R(\CF, 0, 4)$.   By \cite[Lemma~4.5]{PS18},  the natural  map  $\Out_{\CF}(R)  \to \Aut(Z_2)  $   induced by restriction  is injective  on   $O^{p'}(\Out_{\CF} (R))$.   Identify  $\Aut(Z_2)$  with $\GL_2(p)  $  via the basis $x_6, x_5 $. Then by order considerations  the image of  $O^{p'}(\Out_{\CF} (R))$  is $\SL_2(p)  $ and the  image of $\Out_S(R) $  is   the group of upper unitriangular matrices, the image of   an element $\biggl( t,    \begin{pmatrix} \alpha & 0\\0& 1 \end{pmatrix} \biggr) $  of $D\cap  B'$  is the diagonal matrix $\begin{pmatrix}  t^2\alpha^3 &0\\0& t\alpha^3 \end{pmatrix}$.   In particular, $\Out_{\CF}(R)$ acts transitively  on  the non-trivial characters of  $\Irr (Z_2) $. Thus it suffices  to describe the  action  of    $\Stab_{\Out_{\CF}(R) }(\mu) $  on $\Irr( R |  \mu )$  for   any  non-trivial  $\mu  \in \Irr(Z_2) $.   Choose $\mu  \in \Irr(Z_2)$  whose restriction to $Z_1 $  is  trivial.    Then  the image  $\Out_S(R)$   of $S$ in $\Out_{\CF}(R) $  stabilises   $\mu $.  On the other hand, the  stabiliser  in $O^{p'} (\Out_{\CF} (R) ) $  of   any  non-trivial    element of $\Irr(Z_2) $  has a  normal  Sylow $p$-group of  order $p$.    Thus, $\Out_S(R)$  is the unique   Sylow $p$-subgroup  of  $\Stab_{\Out_{\CF}(R) }(\mu)$.   By Lemma~\ref{prop:R non-linear}(iii), $\Out_S(R)$-stabilises every element of  $\Irr( R |  \mu )$. It follows that  $\Out_S(R)$  is a normal $p$-subgroup of   the stabiliser  of $\mu $ and the contribution of $A_2 $ to the  $\sigma_1$-component of   $\w_R(\CF, 0, 4)$  is $0$.  
\end{proof}

\subsection{Result for $Z_3$}

As above, $Z_3$ is elementary abelian of order $p^3$.

\begin{Proposition}\label{prop:z3} Suppose that  $p=5 $ and $P=Z_3$ is $\CF$-centric radical.  Then   $\Out_\CF(P)=\SL_3(5)$ and we have that $\w_P(\CF,0,3)=0$.% if $\Out_\CF(P)=\SL_3(5)$.
\end{Proposition}

\begin{proof}   The first assertion follows by direct calculation using the \texttt{FusionSystems}  package on \textsc{magma} \cite{PS21}.
Write $G=\SL_3(5)$. By \cite[Proposition 5.6(1)]{KLLS}  it suffices to consider $G$-classes of elementary abelian chains of $5$-subgroups. Writing $$S=\left\langle z,a,b  \mid a^p=b^p=z^p=[z,a]=[z,b]=1, [a,b]=z \right\rangle \cong 5^{1+2}_+$$ for a Sylow $5$-subgroup of $G$, there are five $G$-classes of non-trivial elementary abelian $5$-subgroups represented by $$W_1:=Z(S)=\left\langle z \right\rangle, W_2:=\left\langle ab \right\rangle, V_0:=\left\langle z, a \right\rangle, V_1:=\left\langle z,b \right\rangle, V_2:=\left\langle z,ab \right\rangle  $$ which assemble into ten $G$-classes of chains represented by $$(1), (1< W_i) \ (1 \le i \le 2), (1< V_i) \ (0 \le i \le 2), (1 < W_1 < V_i) \ (0 \le i \le 2) \mbox{ and } (1 < W_2 < V_2).   $$
Since $N_G(W_2,V_2)=N_G(W_2) \subseteq N_G(V_2)=N_G(W_1,V_2)$, and $N_G(W_1)=N_G(W_1,V_0)$, there is equality among stabilisers of $$(1 < W_2 < V_2) \mbox{ and } (1 < W_2); (1 < W_1 < V_2) \mbox{ and } (1 < V_2); (1 < W_1) \mbox{ and } (1 < W_1 < V_0) $$ so it suffices to consider the contribution to $\w_P(\CF,0,3)$ of the four $G$-classes of  chains represented by $$(1), (1 < V_0), (1 < V_1) \mbox{ and } (1 < W_1 < V_1).$$
$G$ has two orbits on $\mathbb{F}_5^3$ with stabilisers $G, \ 5^2.\SL_2(5)$ and thus the contribution to $\w_P(\CF,0,3)$ from the trivial chain is $z(kG)=1$ (given by the Steinberg module). The three orbits of $N_G(V_0)$ have stabilisers of shapes $$N_G(V_0)=(C_5 \times C_5) \rtimes \GL_2(5), \ 5^{1+2}_+ \rtimes C_4 \ \mbox{ and } \ \SL_2(5)$$ from which we see that the contribution to $\w_P(\CF,0,3)$ from $(1 < V_0)$ is $-z(k\SL_2(5))=-1$. On the other hand the three orbits of $N_G(V_1)$ have stabilisers of shapes $$N_G(V_1)=(C_5 \times C_5) \rtimes \GL_2(5), \ (C_5 \times C_5) \rtimes \SL_2(5) \  \mbox{ and } \ (C_5 \times C_5) \rtimes C_4$$ and so the contribution to $\w_P(\CF,0,3)$ from $(1 < V_1)$ is $0$. Finally, the four orbits of $N_G(W_1)$ have stabilisers of shapes $$N_G(W_1)=5^{1+2}_+ \rtimes (C_4 \times C_4), \ 5^{1+2}_+ \rtimes C_4, \  C_5 \rtimes C_4 \ \mbox{ and } \ (C_5 \times C_5)\rtimes C_4, $$ from which we see that the chain $(1 < W_1 < V_1)$ also supplies no contribution. The result follows from this.
\end{proof}

\section{Proof  of Theorem~\ref{thm:first_general} and Theorem~\ref{thm:second_general}.}
Let  $p\geq 3 $, let $S$  be a Sylow  $p$-subgroup of $G_2(p) $ and  let  $\CF$  be a  Parker-Semeraro system on $S$.  
By  the main theorem of \cite{PS18},   when $p \geq 5$, then  $\CF$ is  one of   the systems listed in Table~\ref{pstable}; columns~3-6   list the outer automorphism  groups  of   essential  subgroups.   Our table is  an  expanded version of   Table~1 of \cite{PS18}   which does not  explicitly list subsystems  of  $p'$-index.   For  all systems  $\CF$  listed  in Table~1 of \cite{PS18} with $\Gamma_{p'}(\CF) \ne 1 $,  $\Gamma_{p'}(\CF)  $  is cyclic of order $2$ or $3$
 unless $\CF=\CF_7^1(6)$ in which case  $\Gamma_{p'} (\CF) $  is  cyclic of order $6$. Thus  for all  $\CF$  other than   $ \CF_7^1(6)$ for which $\Gamma_{p'}(\CF) \ne 1 $, $O^{p'}(\CF) $ is the unique proper subsystem of $p'$-index    while  $ \Gamma_{p'} (\CF_7^1(6)) $ has  two further  subsystems   which we  denote $ \CF_7^1(6)'$ and  $ \CF_7^1(6)''$ (see \cite[Theorem~7.7]{AKO}).  The  row entries  corresponding to the  subsystems   may be deduced   from the  analysis in \cite{PS18}  and may also be  obtained from the \textsc{magma} algorithms  described in \cite{PS21}.  The last  row  of our table provides the relevant information  for the  $p$-fusion system on  $G_2(p)$, $p \geq 5 $.

\begin{Proposition} \label{prop:mvalues}  Suppose that $p\geq 5 $.  For all $d\geq 0 $,   $\m (\CF,  0, d)  $   is as shown in Table~\ref{pstable2},  (for any value not appearing in the table, the corresponding $\m(\CF, 0, d)=0 $).  \end{Proposition}

\begin{proof}  It suffices to compute the values $\w_U(\mathcal{F},0,d)$ for $U \in \mathcal{F}^{cr}.$   First suppose that $p=5 $. By Proposition \ref{martin} and \cite[Theorem 4.1]{PS18}, $\mathcal{F}^{cr} \subseteq \{Z_3,Q,R,S\}$ and  the relevant automizers may be computed using the \texttt{FusionSystems}  package on \textsc{magma} \cite{PS21}. With this information, for $U \in \{Q,R,S\}$ the values of $\w_U(\mathcal{F},0,d)$ are calculated and listed in Table \ref{table:g25}. For $U=Z_3$, the relevant values are all $0$ by Proposition \ref{prop:z3}.

Next, suppose that $p=7 $ and  let $\CF $ be  a Parker-Semeraro system   on $S$ different from  the fusion system of  $G_2(p) $. By Proposition \ref{martin} and \cite[Theorem 4.1]{PS18}, $\mathcal{F}^{cr} \subseteq \mathcal{W} \cup \{Q,R,S\}$.  Moreover, by Proposition 
\ref{prop:W}, $\w_W(\CF, 0, d)=0$ for  all $d$. Hence it suffices to compute the values $\w_U(\mathcal{F},0,d)$ for $U \in \mathcal{F}^{cr}\cap  \{ Q, R, S\}$.    The  relevant automizers   are read off from Table~\ref{pstable}.

Suppose that  $\CF^{cr}   \cap  \{ Q, R, S\} =\{S\}$. Then   $\m(\CF, 0, d)=\w_S(\CF, 0, d)$ and  by Proposition \ref{prop:S non-linear_general}, $\m(\CF, 0, d)=0$ for $d\leq 3$. The values of
$\m(\CF, 0, d)$ can be computed by using Proposition \ref{weight: S linear_general} and Proposition \ref{weight: S linear} when $d=6$; Proposition \ref{prop:S5} when $d=5$; and Proposition
\ref{prop:S4} when $d=4$.   Suppose that  $\CF^{cr} \cap  \{ Q, R, S\}=\{Q, S\}$, so that  $\m(\CF, 0, d)=\w_Q(\CF, 0, d)+\w_S(\CF, 0, d)$. When $d\neq 3$ and $d\neq 5$, the result follows from Propositions \ref{weight: S linear_general} and  \ref{prop:S4}. Lemma \ref{weight:Q-S} implies that the value of $\m(\CF, 0, 5)$ is equal to contribution of the trivial chain to $\w_Q(\CF, 0, 5)$, and this quantity is computed in Propositions \ref{trivial Q_general} and \ref{doublecover}. We have that $\m(\CF, 0, 3)=\w_Q(\CF, 0, 3)$ and this can be found by using Propositions \ref{weightq} and \ref{weightqS7}. Now suppose $\CF^{cr}\cap  \{ Q, R, S\}=\{R, S\}$, so that  $\m(\CF, 0, d)=\w_R(\CF, 0, d)+\w_S(\CF, 0, d)$. When $d\leq 3$, we have that $\m(\CF, 0, d)=0$. Similar to the previous case, Lemma \ref{weight:Q-S} implies that the value of $\m(\CF, 0, 4)$ is equal to contribution of the trivial chain to $\w_R(\CF, 0, 4)$, and this quantity is computed in Proposition \ref{wR(F,0,4)_general}. The value of $\m(\CF, 0, 5)$ can be computed easily by using Propositions \ref{weight: R} and \ref{prop:S5} and the value of $\m(\CF, 0, 6)$ can be computed by the help of Propositions \ref{weight: S linear_general} and \ref{weight: S linear}. Finally suppose that $\CF^{cr} \cap  \{ Q, R, S\}=\{Q, R, S\}$. For $d\neq 4$ or $d\neq 5$, the values of $\m(\CF, 0, d)$ can be computed as in the previous cases.
Since $\w_R(\CF, 0, 5)=0$, for computing $\m(\CF, 0, 5)$, we can follow the same strategy as in finding $\m(\CF, 0, 5)$ in the $\CF^{cr} \cap  \{ Q, R, S\}=\{Q, S\}$ case. 
Since $\w_Q(\CF, 0, 4)=0$, for computing $\m(\CF, 0, 4)$, we can follow the same strategy as in finding $\m(\CF, 0, 4)$ in the $\CF^{cr}\cap  \{ Q, R, S\}=\{R, S\}$ case.    

Finally suppose  that  $p\geq 5$  and  $\CF =\CF_p$  is the fusion system of $G_2(p) $. By \cite[Theorem 5.9]{PS18}  and  Proposition~\ref{martin}, 
$\mathcal{F}^{cr} =\{Q,R,S\}$.  Moreover,  by the proof of \cite[Lemma 3.5]{PS18} and   \cite[Theorem 5.12]{PS18},
$\Out_{\CF}(Q)\cong  \Out_{\CF}(R) \cong \GL_2(p)$  amd $\Out_{\CF} (S) \cong C_{p-1} \times C_{p-1}$.
By Propositions~\ref{weightw_general}, \ref{prop:S non-linear_general}  and \ref{prop:R non-linear}, $\m(\CF,0, d)  $ is zero unless
 $d \in \{3,4,5,6\}$. First we have $\m(\CF,0,3)=\w_Q(\CF,0,3)=1$ by Proposition \ref{weightq}. Next, $$\m(\CF,0,4)=\w_S(\CF,0,4)+\w_R(\CF,0,4)$$ and by Lemma~\ref{weight:Q-S}, the contribution of the non-trivial chain  to $\w_R(\CF, 0, 4)$ is the negative of $\w_S(\CF,0,4)$. Thus $\m(\CF,0,4)$ is the contribution of the trivial chain  to $\w_R(\CF, 0, 4)$ which is $1$ by Proposition \ref{wR(F,0,4)_general}. When $d=5$, then since $\w_R(\CF,0,5)=0$ by Proposition \ref{weight: R} we have $$\m(\CF,0,5)=\w_S(\CF,0,5)+\w_R(\CF,0,5).$$ As above, by Lemma~\ref{weight:Q-S}, the contribution of the non-trivial chain  to $\w_Q(\CF, 0, 5)$ is the negative of $\w_S(\CF,0,5)$ and thus $\m(\CF,0,5)$ is the contribution of the trivial chain  to $\w_Q(\CF, 0, 4)$ which is $2p+6$ by Proposition \ref{trivial Q_general}. Finally, we have $$\m(\CF,0,6)=\w_S(\CF,0,6)=p^2$$ by Proposition \ref{weight: S linear_general}.

\end{proof} 

\begin{Proposition} \label{prop:wvalues}
Suppose that $p \geq 5 $. Then $\w(\mathcal F, 0)$   is as in   Table~\ref{pstable2}.
\end{Proposition}

\begin{proof}  Lemma 5.10(d) of \cite{PS18}, Table \ref{pstable} and Proposition \ref{martin} give us the following
$$\w(\CF, 0)=\delta_W \frac{(p-1)^2}{|\Out_{\CF}(S)|} z(\Out_{\CF}(W))+\delta_Z z(\Out_{\CF}(Z_3))+\delta_Q z(\Out_{\CF}(Q))+ \delta_R z(\Out_{\CF}(R))+z(\Out_{\CF}(S))$$
where $\delta_P=1$ if $P$ is $\mathcal F$-centric radical, and $0$ otherwise.

%{\bf[R:  Here  also  we need not give  the details for every case (of course we work them  out  for  ourselves)]}  
When $p=5$ we have $\delta_W=0$ and $\delta_Q=\delta_R=1$;  $\delta_Z=0$ unless $\CF=\CF_5^0$ or $\CF=\CF_5^2$ where we have $z(\Out_{\CF}(Z_3))=z(\SL_3(5))=1$. Thus,
\begin{equation*}
\begin{split}
\w(\CF, 0)&=\begin{cases}
		\begin{split}	
           1+4+14+16 =35, & \qquad\text{if $\CF=\CF_5^0$}. \\           
   4+12+16 = 32, & \qquad\text{if $\CF=\CF_5^1$} \\      
            2+6+8= 16, & \qquad \text{if $\CF=O^{5'}(\CF_5^1)$}\\
                1+4+30+16 = 51, & \qquad\text{if $\CF=\CF_5^2$.} \\    
 \end{split}		 
		 \end{cases}
		 \end{split}
		 \end{equation*}

For all $p \ge 5$ we have,
$$\w(\CF_p, 0)=z(\Out_{\CF_p}(R))+z(\Out_{\CF_p}(Q))+z(\Out_{\CF_p}(S))=2z(k\GL_2(p))+z(k C_{p-1}\times C_{p-1})=p^2-1.$$

The rest of the computation (when $p=7$) can be done easily given that $z(k \SL_2(p))=1$, $z(k \GL_2(p))=p-1$ and $z(kC_3 \times 2 \mathfrak S_7)=27$.

\end{proof}

Proof of Theorem~\ref{thm:first_general}:
\begin{proof}  Suppose first that $p\geq 5$.  We  have $|S|=p^6$. Since $Q/\Phi(Q)$ has rank  $4$, $S/\Phi(S)$ has rank  $2$, and  $S$ has no abelian quotient groups of order $p^5$, the sectional rank of $S$ is   $4$.     By Proposition~\ref{prop:S non-linear_general}, the number of  conjugacy classes of  $S$ equals  $p^3 +2p^2 -p -1 $.  Further,    $[S,S]= Z_4 $ is the direct product of a cyclic group of  order $p$ and  an extra-special group of order  $p^2$ (see \cite[Section 3]{PS18}),  hence the  number of  conjugacy classes of $[S,S]= p(p^2 +p-1) $.  By Proposition~\ref{prop:S non-linear_general},   the   integer $r$ of Conjecture~\ref{conj:weight} (iv) equals $1$.
Now Theorem~\ref{thm:first_general}  follows by comparing these values  to the values  given  in Table~\ref{pstable2}.

Now suppose that $p=3 $. By \cite[Theorem 7.2]{PS18}, $\CF$ is either realised by   $G_2(3)$ or $\Aut(G_2(3))$.  There are two $\CF$-classes ($Q_1^\CF$ and $Q_2^\CF$) of $\CF$-centric radical subgroups in the former fusion system which are fused in the latter system.   Using \textsc{magma} one computes that  the sectional rank of  $S$ is $4$, the number of conjugacy classes   of  $S$ is  73, that of $[S,S] $  is  $3^3$, the  values of  $m (\CF, 0, d)$  are as  in Table \ref{table:g23} and the number of weights of $\CF$  equals $8$  (respectively $7$) if   $\CF$ is realised by   $G_2(3)$  (respectively $\Aut(G_2(3))$). Theorem~\ref{thm:first_general}  is  an easy check.
\end{proof}

Proof of Theorem~\ref{thm:second_general}:

\begin{proof}  The  character degrees  of $G_2(3)$, $\Aut(G_2(3))$,  $\HN$, $\BM$, Aut$(\HN)$ and $Ly$ are stored in \textsc{gap}, and the  character degrees  of the principal  $3$, $5$ and $7$-blocks  are readily recovered from \textsc{gap} using the $\texttt{PrimeBlocks}$ command. For
 $p  \geq 5$ the  irreducible characters of $G_2(p)$  were described   by Chang and  Ree \cite{ChRee}  and the irreducible character degrees   in the principal block can be recovered by  noting that all irreducible characters  of  $G_2(p) $  except the Steinberg character  (of degree $p^6$)   lie in the principal  $p$-block.
The claim  follows  by comparing with    Table~\ref{pstable2}   and  Table \ref{table:g23}.
\end{proof}

\bigskip   

\begin{table}
\caption{Parker--Semeraro systems $\mathcal F$ for $p \ge 5$} \label{pstable}
\begin{center}
 \begin{tabular}{||c c c c c c c||} 
 $\mathcal F$ & $p$ & $\Out_{\mathcal F}(W)$ & $\Out_{\mathcal F}(R)$ & $\Out_{\mathcal F}(Q)$ & $\Out_{\mathcal F}(S)$ & Group \\ [0.5ex] 
  $\mathcal F^0_5$ & 5 & - & $\GL_2(5)$ & $2^. \mathfrak{A}_6.4$ & $4 \times 4$  & Ly \\ 
  $\mathcal F^1_5$ & 5 & - & $\GL_2(5)$ & $4 \circ 2^{1+4}_-.\Frob(20)$ & $4 \times 4$ &  $ \Aut($HN$)$ \\ 
 $ O^{5'}(\mathcal F^1_5)$ & 5 & - & $4\circ \SL_2(5)$ & $ 2^{1+4}_-.\Frob(20)$ & $4 \times 2$  &  $\HN$ \\ 
$\mathcal F^2_5$ & 5 & - & $\GL_2(5)$ & $2^{1+4}_-. \mathfrak{A}_5.4$ & $4 \times 4$ & $\BM$ \\ 
  $\mathcal F^0_7$ & 7 & - & $\GL_2(7)$ & $3 \times 2 \mathfrak{S}_7$ & $6 \times 6$ & - \\ 
 $\mathcal F^1_7(1_1)$ & 7 & $\SL_2(7)$ & - & - & 6 & -  \\
 $\mathcal F^1_7(2_1)$ & 7 & $\SL_2(7)$ & - & - & 6 &  -  \\
 $\mathcal F^1_7(2_2)$ & 7 & $\SL_2(7)$ & - & - & 6 & -  \\
 $\mathcal F^1_7(2_3)$ & 7 & $\SL_2(7)$ & - & - & $6 \times 2$ &  -  \\
 $ O^{7'}(\CF^1_7(2_3))$ & 7 & $\SL_2(7)$ &-&-& $6$ &- \\
 $\mathcal F^1_7(3_1)$ & 7 & $\SL_2(7)$ & - & - & 6 &  -  \\
 $\mathcal F^1_7(3_2)$ & 7 & $\SL_2(7)$ & - & - & 6 &  -  \\
 $\mathcal F^1_7(3_3)$ & 7 & $\SL_2(7)$ & - & - & 6 & -  \\
 $\mathcal F^1_7(3_4)$ & 7 & $\SL_2(7)$ & - & - & $6 \times 3$ &  -  \\
  $ O^{7'}(\CF^1_7(3_4))$ & 7 & $\SL_2(7)$ &-&-& $6$ & -\\
 $\mathcal F^1_7(4_1)$ & 7 & $\SL_2(7)$ & - & - & 6 &  -  \\
 $\mathcal F^1_7(4_2)$ & 7 & $\SL_2(7)$ & - & - & 6 &  - \\
 $\mathcal F^1_7(4_3)$ & 7 & $\SL_2(7)$ & - & - & $6 \times 2$ &  -  \\
  $ O^{7'}(\CF^1_7(4_3))$ & 7 & $\SL_2(7)$ &-&-& $6$ & -\\
 $\mathcal F^1_7(5)$ & 7 & $\SL_2(7)$ & - & - & 6 &  -  \\
 $\mathcal F^1_7(6)$ & 7 & $\SL_2(7)$ & - & - & $6 \times 6$ & -  \\
 $ \CF^1_7(6)'$ & 7 & $\SL_2(7)$ &-&-& $6\times 2$ &-\\
 $\CF^1_7(6)''$ & 7 & $\SL_2(7)$ &-&-& $6\times 3$ &-\\
 $ O^{7'}(\CF^1_7(6))$ & 7 & $\SL_2(7)$ &-&-& $6$ &-\\
 $\mathcal F^2_7(1)$ & 7 & $\SL_2(7)$ & $\SL_2(7).2$ & - & $6 \times 2$ & -  \\
 $\mathcal F^2_7(2)$ & 7 & $\SL_2(7)$ & $\SL_2(7).2$ & - & $6 \times 2$ &  -  \\
 $\mathcal F^2_7(3)$ & 7 & $\SL_2(7)$ & $\GL_2(7)$ & - & $6 \times 6$ & -  \\
   $ O^{7'}(\CF^2_7(3))$ & 7 & $\SL_2(7)$ &$\GL_2(7)$&-& $6\times 2$ &-\\
 $\mathcal F^3_7$ & 7 & $\SL_2(7)$ & - & $\GL_2(7)$ & $6 \times 6$ &  -  \\
 $\mathcal F^4_7$ & 7 & $\SL_2(7)$ & - & $3 \times 2 \mathfrak{S}_7$ & $6 \times 6$ & -  \\
 $\mathcal F^5_7$ & 7 & $\SL_2(7)$ & $\GL_2(7)$ & $\GL_2(7)$ & $6 \times 6$ &  -  \\
 $\mathcal F^6_7$ & 7 & $\SL_2(7)$ & $\GL_2(7)$ & $3 \times 2 \mathfrak{S}_7$ & $6 \times 6$ & $M$\\
  %$\mathcal F^7_7=\CF_7$ & 7 & $-$ & $\GL_2(7)$ & $\GL_2(7)$ & $6 \times 6$ & $G_2(7)$ \\
  $\CF_p$ & $p$ & $-$ & $\GL_2(p)$ & $\GL_2(p)$ & $(p-1) \times (p-1)$ &  $G_2(p)$
\end{tabular}
\end{center}
\end{table}

\begin{table}
\caption{Values for $\m(\CF, 0, d)$ and $\w(\CF,0)$ for $p \ge 5$} \label{pstable2}
\begin{center}
 \begin{tabular}{||c c c c c c c c c c||} 
 $\mathcal F$ & $p$ & $\w(\CF,0)$ & $\m(\CF, 0, 6)$ & $\m(\CF, 0, 5)$ & $\m(\CF, 0, 4)$ & $\m(\CF, 0, 3)$ &  Group \\ [0.5ex] 
 $\CF_5^0$ & $5$ &$35$& $25$ & $17$ & $1$ & $5$ & $\Ly$ \\ 
   $\CF_5^1$ & $5$ &$32$& $25$ & $31$ & $1$ & $6$ & $\Aut(\HN)$ \\
  $O^{5'}(\CF_5^1)$ & $5$ & $16$ & $20$ & $20$ & $2$ & $3$ & $\HN$ \\

 $\CF_5^2$ & $5$ &$51$& $25$ & $44$ & $1$ & $9$ & $\BM$ \\

   $\mathcal F^0_7$  &7&$69$& 49 & $20$ & $4$ & $6$  & - \\ 
 $\mathcal F^1_7(1_1)$  &7&$12$& $14$ & 110 & 7 & 0  & -  \\
 $\mathcal F^1_7(2_1)$  &7&$12$& $14$ & 110 & 7 & 0  & -  \\
 $\mathcal F^1_7(2_2)$  &7&$12$& $14$ & 110 & 7& 0  & -  \\
 $\mathcal F^1_7(2_3)$  &7&$15$& $19$ & 56 & 14 & $0$  & -  \\
 $ O^{7'}(\CF^1_7(2_3))$  &7&12& $14$ & 110 & 7 & $0$  & -  \\
 $\mathcal F^1_7(3_1)$  &7&$12$& $14$ & 110 & 7 & 0  & -  \\
 $\mathcal F^1_7(3_2)$  &7&$12$& $14$ & 110 & 7 & 0  & -  \\
 $\mathcal F^1_7(3_3)$  &7&$12$& $14$ & 110 & 7& 0 & -  \\
 $\mathcal F^1_7(3_4)$  &7&$20$& $26$ & 67 & 5 & $0$  & -  \\
 $ O^{7'}(\CF^1_7(3_4))$  &7&12& $14$ &110 & 7 & $0$  & -  \\
 $\mathcal F^1_7(4_1)$  &7&$12$& $14$ & 110 & 7 & 0  & -  \\
 $\mathcal F^1_7(4_2)$  &7&$12$& $14$ & 110 & 7& 0  & - \\
 $\mathcal F^1_7(4_3)$  &7&$15$& $19$ & 56& 14 & $0$  & -  \\
 $ O^{7'}(\CF^1_7(4_3))$ &7&12& $14$ &110& 7 & $0$  & -  \\ 
 $\mathcal F^1_7(5)$  &7&$12$& $14$ & 110 & 7 & 0  & -  \\
 $\mathcal F^1_7(6)$  &7&$37$& $49$ & 41& 10 & $0$  & -  \\
 $\mathcal F^1_7(6)'$  &7&15& $19$ & 56& 14 & $0$  & -  \\
  $\mathcal F^1_7(6)''$  &7&20& $26$ & 67& 5 & $0$  & -  \\
   $O^{7'} (\CF^1_7(6))$  &7&12& $14$ & 110& 7 & $0$  & -  \\ 
 $\mathcal F^2_7(1)$  &7&$17$& $19$ & $56$ & 12 & $0$  & -  \\
 $\mathcal F^2_7(2)$  &7&$17$& $19$ & $56$ & 12 & $0$  & -  \\
 $\mathcal F^2_7(3)$  &7&$43$& $49$ & $36$ & 4 & $0$  & -  \\
 $O^{7'}(\CF^2_7(3))$  &7&15& $19$ & $56$ & 12 & $0$  & -  \\
 $\mathcal F^3_7$  &7&$43$& $49$ & 20 & $10$ & $1$  & -  \\
 $\mathcal F^4_7$  &7&$64$& $49$ & 36 & $10$ & $6$  & -  \\
 $\mathcal F^5_7$  &7&$49$& $49$ & $20$ & $4$ & $1$  & -  \\
 $\mathcal F^6_7$  &7&$70$& $49$ & $36$ & $1$ & $6$  & $M$\\
  $\mathcal F^7_7=\CF_7$  &7&$48$& $49$ & $20$ & $1$ & $1$  & $G_2(7)$ \\
  $\mathcal F_p$  &$p$&$p^2-1$&$p^2$&$2p+6$&$1$&$1$ & $G_2(p)$ \\
\end{tabular}
\end{center}
\end{table}

\begin{table}[]

\renewcommand{\arraystretch}{1.4}
\centering
\caption{Values for $\w_P(\mathcal{F},0,d)$ for $\mathcal{F} = \mathcal{F}^0_5,\mathcal{F}_5^1,O^{5'}(\mathcal{F}_5^1), \mathcal{F}_5^2$ when $p=5$}
\label{table:g25}
\begin{tabular}{|c|c|c|c|c|}
\hline
   $P$   & $\w_P(\mathcal{F},0,3)$  & $\w_P(\mathcal{F},0,4)$  & $\w_P(\mathcal{F},0,5)$ & $\w_P(\mathcal{F},0,6)$    \\ \hline
\hline
$S$      & $-$         & $8,8,10,8$ & $25,25,38,25$         & $25,25,20,25$    \\ \hline
$Q$ & $5,6,3,9$ & $-$ & $-8,6,-18,19$ & $-$ \\ \hline
$R$ & $-$ & $-7,-7,-8,-7$ & $0,0,0,0$ & $-$ \\ \hline 
$Z_3$ & $0,-,-,0$ & $-$ & $-$ & $-$ \\ \hline 
\end{tabular}

\end{table}

\begin{table}[]
\renewcommand{\arraystretch}{1.4}
\centering
\caption{Values for $\w_P(\mathcal{F},0,d)$ for $\mathcal{F} = \CF_3, \CF_3:2$ when $p=3$}
\label{table:g23}
\begin{tabular}{|c|c|c|c|}
\hline
   $P$   &  $\w_P(\mathcal{F},0,4)$  & $\w_P(\mathcal{F},0,5)$ & $\w_P(\mathcal{F},0,6)$    \\ \hline
\hline
$S$      & $1,2$         & $24,15$ & $15,18$        \\ \hline
$Q_1$ & $0,0$ & $-9,-9$ & $-$ \\ \hline
$Q_2$ & $0$,$-$ & $-9$,$-$ & $-$ \\ \hline
\end{tabular}
\end{table}

\iffalse

\begin{table}[]
\renewcommand{\arraystretch}{1.4}
\centering
\caption{$\w_P(\mathcal{F},0,d)$ for all $d \ge 0$}
\label{table:g2p}
\begin{tabular}{|c|c|c|c|c|}
\hline
   $P$   & $3$  & $4$  & $5$ & $6$    \\ \hline
\hline
$S$         & $-$         & $p+3$ & $p^2$         & $p^2$    \\ \hline
$Q$ & $1$ & $-$ & $2p+6-p^2$ & $-$ \\ \hline
$R$ & $-$ & $-(p+2)$ & $0$ & $-$ \\ \hline
$\m(\mathcal{F},0,d)$ & $1$ & $1$ & $2p+6$ & $p^2$  \\ \hline
\end{tabular}
\end{table}

\fi

{\it Acknowledgements.} We thank Martin van Beek for help  on the proof of Proposition~\ref{martin}. J. Semeraro gratefully acknowledges funding from the
UK Research Council EPSRC for the project EP/W028794/1. P. Serwene thanks Klaus Lux for substantial help with \textsc{gap}. \.I. Tuvay was supported by the Scientific and Technological Research Council of Turkey (T\"ubitak) through the research program Bideb-2219 during her visit to University of Manchester.
%%%%%%%%%%%%%%%%%%%%%%%%%%%%%%%%%%%%%%%%%%%%%%%%%%%%%%%%%%%%%%%%%%

\end{document}